\documentclass[reqno]{amsart}
\usepackage{caption}
\usepackage{amsmath,amssymb,amsthm}
\usepackage{amsmath}
\usepackage{amsfonts}
\usepackage{float}
\usepackage{mathtools}
\usepackage{mathrsfs}  
\usepackage{youngtab}
\usepackage[pdftex]{graphicx}
\usepackage{tikz}
\usetikzlibrary{decorations.markings,arrows,automata,positioning}
\tikzset{partition/.style={fill,circle,inner sep=1pt}}
\usetikzlibrary{positioning}
\usetikzlibrary{decorations.pathreplacing}

\usetikzlibrary{matrix,arrows}
\usetikzlibrary{calc}
\tikzset{partition/.style={fill,circle,inner sep=1pt},
         part/.style={baseline=0,scale=0.5,bend left=45},
         partlabel/.style={below}}
\usetikzlibrary{shapes,arrows}
\usetikzlibrary{snakes}
\tikzstyle{pnt}=[draw,ellipse,fill,inner sep=1pt]

\tikzstyle{opnt}=[draw,ellipse,inner sep=1pt]
\tikzstyle{opnt}=[ ]
\tikzstyle{pntt}=[draw,ellipse,fill,inner sep=0.5pt]
\tikzstyle{point}=[draw,ellipse,fill,inner sep=2pt]

\usepackage{color}

\newcommand{\Hom}{\operatorname{Hom}}

\newcommand{\jmp}{\operatorname{jmp}}

\newcommand{\cent}{\operatorname{Centralizer}}

\DeclareMathOperator{\Tr}{Tr}

\usepackage{amssymb}

\newtheorem{theorem}{Theorem}[section]
  \newtheorem{lemma}[theorem]{Lemma}
  \newtheorem{palgorithm}[theorem]{Separation of Variables (SOV) Approach}

  \newtheorem{corollary}[theorem]{Corollary}
  	\newtheorem{definition}[theorem]{Definition} 
  	\theoremstyle{definition}
  	\newtheorem{example}[theorem]{Example} 
  \theoremstyle{remark}

\allowdisplaybreaks

\subjclass[2000]{05C25, 05E40, 05C38, 13F20}

\begin{document}
\title{The efficient computation of Fourier transforms on semisimple algebras}

\author{David Maslen}
\address{HBK Capital Management, New York, NY 10036}
\email{david@maslen.net}

\author{Daniel N. Rockmore}
\address{Departments of Mathematics and Computer Science, Dartmouth College, Hanover, NH 03755}
\email{rockmore@math.dartmouth.edu}
\thanks{The second author was partially supported by AFOSR Award FA9550-11-1-0166 and the Neukom Institute for Computational Science at Dartmouth College}

\author{Sarah Wolff}
\address{Department of Mathematics, Denison University, Granville, OH 43023}
\email{wolffs@denison.edu}
\thanks{The third author was partially supported by an NSF Graduate Fellowship.}

\subjclass[2000]{To be filled in.}

\date{\today}


\keywords{Fast Fourier Transform, Bratteli diagram, path algebra, quiver}

\begin{abstract}
We present a general diagrammatic approach to the construction of efficient algorithms for computing a Fourier transform on a semisimple algebra. This extends  previous work wherein we derive best estimates for the computation of a Fourier transform for a large class of  finite  groups. We continue to find efficiencies by exploiting a  connection between Bratteli diagrams and the derived path algebra and  construction of Gel'fand-Tsetlin bases. Particular results include highly efficient algorithms for the Brauer, Temperley-Lieb algebras, and Birman-Murakami-Wenzl algebras.  
\end{abstract}

\maketitle

\section{Introduction}\label{in}
In this paper we take up the problem of the efficient computation of a Fourier transform on a finite-dimensional complex semisimple algebra. The work herein is born of earlier generalizations framing the classical ``fast Fourier transform" (FFT)  within the context of a finite group algebra. In this setting, an FFT is  efficient computation of a change of basis in the group algebra, from a basis of point masses on group elements  to a basis of irreducible matrix elements \cite{beth,clausen,MaslenNotices}. The important case of cyclic groups  finds its origins in work of  Gauss (see \cite{cooley,RockmoreIEEE}) and  the family of efficient algorithms for computing the Fourier transform on a finite abelian group (usually known as the {\em discrete Fourier transform} or DFT), is collectively referred to as  ``the FFT". The FFT has been and continues to be the engine of the world of digital signal processing (see e.g., the classic texts \cite{bracewell,rao} as well as references in \cite{RockmoreIEEE}). Applications of nonabelian Fourier transforms (i.e., when the group is nonabelian) can be found in a range of domains including voting theory, filter design, coding theory, and domain reduction for solving PDEs \cite{diaconisspec, kondortracking, munthekaas, wood,rocksurvey}. 


As a linear change of basis in the complex group algebra $\mathbb{C}[G]$ (for a finite group $G$), a Fourier transform algorithm has an obvious upper bound of $|G|^2$ complex operations.\footnote{We use here a standard definition of operation count as a complex addition and multiplication. In various places we may break out the number of additions and multiplications separately, but this will have no effect on the ``big $O$" kinds of results we present here.} Motivated mainly by the many important applications, there is now a large and important body of work -- whose origin is usually traced to the fundamental paper of Cooley and Tukey \cite{cooleytukey} -- showing  that for the case of abelian groups,  a master algorithm for computing the Fourier transform has computational complexity $O(|G|\log |G|)$ \cite{diaconis-fft}. Careful  counting and fine-tuning produce an explicit upper bound of   $8|G|\log_2|G|$ \cite{baumclausentietz}. 

The abelian result has made an $O(|G|\log^c |G|)$ (for any constant $c$, independent of $G$) upper bound  the standard benchmark for high efficiency. While conjectured for an arbitrary finite group, to date this  has been achieved only for several explicit families  in addition to abelian groups. These include the symmetric groups and their wreath products \cite{maslen}, abelian extensions \cite{rockmoreabext}, and supersolvable groups \cite{baum}.  However, recent progress has produced very efficient -- if not to the gold standard -- algorithms for the general linear groups over finite fields along with the Weyl groups of type $B_n$ and $D_n$ \cite{sovi,sovII}.  The groups $SL_2(p)$ are a particularly interesting and thorny  special case, as an $O(|SL_2(\mathbb{F}_p)|\log |SL_2(\mathbb{F}_p)|)$ algorithm could produce an effective fast matrix multiply algorithm \cite{lafferty1992,MR-duco}. The dependence of efficiency on the type of group is implicit in the definition of {\em the complexity of a group}, denoted $C(G)$ for a finite group $G$, and defined as the least upper bound over all choices of complete sets of inequivalent irreducible representations for computing the Fourier transform on $G$. 

A complex finite group algebra is a specific example of a finite-dimensional semisimple algebra. As per the group algebra case, the Fourier transform of a complex semisimple algebra $A$ is a change of basis from some preferred basis to a basis given by irreducible matrix elements. The complexity of $A$, $C(A)$ is the least upper bound of an algorithm effecting such a map and thus bounded above a priori by $\dim(A)^2$. The work presented in this paper, aiming to reduce this naive upper bound,  is both motivated as a next ``natural" step in  algebraic FFT work (see also extensions to the semigroup case \cite{MalandroRockmore,Malandro1,Malandro2}) as well as by a particular application:  the study of a certain random walk on the Birman-Murakami-Wenzl (BMW) algebra  \cite{wolff}. The usefulness of Fourier analysis for studying random walks on finite groups and algebras is well known (see e.g., \cite{diaconisspec, diaram}) and this paper thus connects with that literature as well. 

Herein we show how the {\em separation of variables} (SOV) approach \cite{sovi,sovII} for efficient group Fourier transforms (used to such great effect in the $S_n$, $D_n$, $B_n,$ and $GL_n(q)$ cases) can be applied to derive efficient algorithms in the general semisimple algebra setting. The SOV approach takes advantage of an isomorphism between the path algebra  associated to the Bratelli diagram attached to a subgroup tower to uncover dependencies and redundancies in the calculation of the Fourier transform.

The natural extension of the finite group case would be to find  algorithms for computing the Fourier transform for a complex semisimple algebra $A$ in 
\linebreak
$O(\dim(A) \log^c\dim(A))$ operations. Herein we show that application of the SOV approach to the Brauer, BMW, and Temperley-Lieb algebras produces in the following bounds.

\begin{theorem}\label{brauerthm}
Let $\mathcal{B}r_n$ denote the $(2n-1)!!$-dimensional Brauer algebra. Then 
$$C(\mathcal{B}r_n)\leq (4n^2-n+4)\dim(\mathcal{B}r_n)\sim O(\dim(\mathcal{B}r_n) \log(\dim (\mathcal{B}r_n))^2)
.$$
\end{theorem}

\begin{theorem}\label{BMWthm} Let $\mathcal{BMW}_n$ denote the $(2n-1)!!$-dimensional BMW algebra. Then 
$$C(\mathcal{BMW}_n)\leq (4n^2-n+4)\dim(\mathcal{BMW}_n)\sim O(\dim(\mathcal{BMW}_n) \log(\dim (\mathcal{BMW}_n))^2).$$
\end{theorem}

\begin{theorem}\label{templiebthm}
Let $\mathcal{T}_n$ denote the Temperley-Lieb algebra, with dimension the $n$th Catalan number. Then 
$$C(\mathcal{T}_n)\leq  \frac{n^3+9n^2+8n-12}{6}\dim(\mathcal{T}_n)\sim O(\dim(\mathcal{T}_n) \log(\dim (\mathcal{T}_n))^3) .$$

\end{theorem}

These theorems could be summarized as saying that the BMW, Brauer, and Temperley-Lieb algebras admit FFTs.

In Section \ref{prelim} we provide the necessary background for our results, defining the Fourier transform on a semisimple algebra, adapted representations, Bratteli diagrams, and Gel'fand Tsetlin bases. In Section \ref{sepvarstate} we introduce and extend the main tools of the SOV approach of \cite{sovII}, providing the definitions of the subsets and quivers that enable the path-counting utilized in the proofs of Theorems \ref{brauerthm}, \ref{BMWthm}, and \ref{templiebthm}. In Section \ref{mainresults} we first provide background and definitions of the Brauer, BMW, and Temperley-Lieb algebras, then prove the complexity results of Theorems \ref{brauerthm}, \ref{BMWthm}, and \ref{templiebthm} using the extended SOV approach. We also provide a general result for semisimple algebras with special subalgebra structure. We conclude in Section \ref{conclusion} with further directions and questions.

This paper necessarily relies on the earlier separation of variables work \cite{sovi,sovII}. It is (regretably) somewhat technical and for reasons of length we cannot reproduce it here in its entirety. The interested reader should see \cite{sovII} for the details of the quiver formulation.

\section{Background}\label{prelim}

\subsection{The Fourier transform of a  semisimple algebra}\label{galgebra}
The usual Fourier transform on a finite group, defined using matrix representations, may be viewed as a special case of a Fourier transform on a semisimple algebra. We work here exclusively in the context of complex semisimple algebras. Recall that a complex algebra is {\em simple} if it is isomorphic to a complex matrix algebra and {\em semisimple} if isomorphic to a finite direct sum of simple algebras. A complex representation of an algebra is an algebra homomorphism $\rho: A\rightarrow M_d(\mathbb{C}),$ where $M_{d}(\mathbb{C})$ denotes the complex algebra of $d\times d$ matrices with entries in $\mathbb{C}$. We call $d$ the {\em{dimension}} of $\rho$.

Results here assume complex representations, unless spelled out otherwise, although most results go through more generally. For necessary background on the representation theory of semisimple algebras we refer the reader to \cite{ram}.



\begin{definition}\label{alg} 
Let $A$ be a semisimple algebra with basis $\{a_i\}_{i\in I}$ and let $\displaystyle f=\sum_{i\in I}f(a_i)a_i$ be the expansion of a given element of $A$ in terms of the basis $\{a_i\}_{i\in I}$.
\begin{itemize}
\item[(i)] Let $\rho$ be a matrix representation of $A$. The \textbf{Fourier transform of} $f$ \textbf{at} $\rho$, denoted $\hat{f}(\rho)$, is the matrix sum
$$\hat{f}(\rho)=\sum_{i\in I} f(a_i)\rho(a_i).$$
\item[(ii)] Let $R$ be a set of matrix representations of $A$. The \textbf{Fourier transform of} $f$ \textbf{on} $R$ is the direct sum of Fourier transforms of $f$ at the representations in $R$:
$$\mathcal{F}_R(f)=\bigoplus_{\rho\in R} \hat{f}(\rho)\in\bigoplus_{\rho\in R} M_{\dim\rho}(\mathbb{C}).$$
 
\end{itemize}
\end{definition}

When we compute the Fourier transform for  a complete set of inequivalent irreducible representations $R$ of $A$  we refer to the calculation as the \textbf{computation of a Fourier transform on $A$} (with respect to $R$).  Notice that this is equivalent to the calculation of the change of basis from $\{a_i\}_{i\in I}$ to the explicit basis given by the evaluation of the matrix elements $\rho^i_{jk}(a_\ell)$ (see Lemma~\ref{equivcomp}). Thus, the definitions depend on explicit choices of bases, both in the initial expansion as well as the target.

\begin{example} When $A=\mathbb{C}[G]$, the complex group algebra of a finite group $G$,with $a_i$ equal to the indicator function that is $1$ on the $i$th element of $G$ and $0$ elsewhere,  Definition \ref{alg} gives the usual definition of the Fourier transform of a function on $G$ \cite{rockmassurvey,sovi,sovII}. Elements of $\mathbb{C}[G]$ are in one-to-one correspondence with complex-valued functions on G, and the Fourier transform of $f:G\rightarrow\mathbb{C}$ at a matrix representation $\rho$ of $G$ is $$\hat{f}(\rho)=\sum_{s\in G} f(s)\rho(s).$$ 
\end{example}

\begin{definition}\label{complexdef} Let $A$ be  a semisimple algebra with basis $\{a_i\}_{i\in I}$ and let $R$ be a set of matrix representations of $A$.
\begin{itemize}
\item[(ii)] Let $+_A(R)$ (respectively, $\times_A(R)$) denote the minimum number of complex arithmetic additions (resp., multiplications)  needed to compute the Fourier transform of $f$ on $R$ via a straight-line program\footnote{A \textbf{straight-line program} is a list of instructions for performing the operations $\times, \div, +, -$ on inputs and precomputed values \cite{algebraiccomplexity}.} for an arbitrary $ f=\sum_{i\in I}f(a_i)a_i$. The \textbf{arithmetic complexity} of a Fourier transform on $R$, denoted $T_A(R)$, is given by $$T_A(R)=\max{(+_A(R), \times_A(R))}.$$
\item[(ii)] The \textbf{complexity of the algebra} $A$, denoted $C(A)$, is given by
$$C(A):=\min_R \{T_A(R)\},$$
where $R$ varies over all complete sets of inequivalent irreducible matrix representations of $A$.
\item[(iii)] The \textbf{reduced complexity}, denoted $t_A(R)$, is given by $$t_A(R)=\frac{1}{\dim(A)}T_A(R).$$
\end{itemize}
\end{definition}
Let $\rho_1,\dots,\rho_m$ be a complete set of inequivalent irreducible matrix representations of an algebra $A$ of dimensions $d_1,\dots,d_m,$ respectively. Direct computation of a Fourier transform would require at most $\dim(A)\sum d_i^2=\dim(A)^2$  arithmetic operations. Rewriting, for a direct computation we have 
$$C(A)\leq T_G(A)\leq \dim(A)^2.$$

\textbf{Fast Fourier transforms} (FFTs) are algorithms for computing Fourier transforms that improve on this naive upper bound.  A priori, the number of operations needed to compute the Fourier transform may depend on the specific representations used. 
\subsection{Fourier Inversion}
A complete set $R$ of inequivalent irreducible matrix representations of a semisimple algebra $A$ determines a basis for $A$ (via the irreducible matrix elements) and in this case the Fourier transform is an algebra isomorphism from $A$ to a direct sum of matrix algebras. We recover $f$ through the Fourier inversion formula, Theorem \ref{planch} below.
\begin{definition} For $A$ a semisimple algebra, a \textbf{trace function} on $A$ is a $\mathbb{C}$-linear function $\tau:A\rightarrow\mathbb{C}$ such that for all $a,b\in A$,
$$\tau(ab)=\tau(ba).$$
\end{definition}
A trace function $\tau$ gives rise to a symmetric bilinear form $\langle \cdot,\cdot\rangle_\tau:A\times A\rightarrow \mathbb{C}$ via $$\langle a,b\rangle_\tau=\tau(ab),$$
for $a,b\in A$.

By linearity the usual trace function on $M_d(\mathbb{C})$ is unique up to multiplication by a constant. Hence, for any trace $\tau$ on $A$ and set $R$ of inequivalent irreducible representations of $A$, there exist constants $t_\rho\in\mathbb{C}$ such that:
$$\tau=\sum_{\rho\in R}t_\rho \Tr(\rho(a)).$$

\begin{theorem}[Fourier Inversion]\label{planch}
Let $A$ be a semisimple algebra with basis $\{a_i\}_{i\in I}$ and $\tau$ a nondegenerate trace on $A$.  Let $\{a_i^*\}$ be the dual basis to $\{a_i\}$ with respect to the trace form $\langle \cdot,\cdot\rangle_\tau$. Then
\begin{equation}
f(a_i)=\sum_{\rho} t_\rho \Tr(\hat{f}(\rho)\rho(a_i^*)).
\end{equation}
\end{theorem}

Thus, the Fourier transform of $f$ on $A$ with respect to a complete set of inequivalent irreducible matrix representations $R$ of $A$ is an algebra isomorphism 
$$ A\xrightarrow{\;\;\;\mathcal{F}_R\;\;\;}\bigoplus_{\rho\in R} M_{\dim(\rho)}(
\mathbb{C}).$$ 
\begin{definition}
For $R$ a complete set of inequivalent irreducible matrix representations of $A$, the inverse image under the Fourier transform $\mathcal{F}_R$ of the natural basis of $\bigoplus_{\rho\in R} M_{\dim(\rho)}(\mathbb{C})$ is the \textbf{dual matrix coefficient basis for} $A$ \textbf{associated to} $R$. 

\end{definition}

\begin{lemma}[e.g. \cite{clausen, maslen}]\label{equivcomp}  The computation of the Fourier transform of $ f=\sum_{i\in I}f(a_i)a_i$ on $A$ with respect to a complete set of irreducible  matrix representations $R$ is equivalent to computation (rewriting) of
$$\sum_{i\in I} f(a_i)a_i,$$ relative to the dual matrix coefficient basis for $R$.
\end{lemma}


\subsection{Bratteli diagrams and quivers}\label{gel}
The computational methodology that we present here is a recasting of a divide-and-conquer (or when viewed from the bottom up, a dynamic programming approach) for computing the Fourier transform in terms of {\em graded quivers}, which is an elaboration of the path algebras derived from Bratteli diagrams.  This is a natural extension of the work in \cite{sovII}. Herein we give the necessary definitions and extensions of the needed lemmas. The interested reader should see the original paper for details. 

\begin{definition} For a subalgebra $B$ of a semisimple algebra $A$, a complete set $R$ of inequivalent irreducible matrix representations of $A$ is $\mathbf{B}$\textbf{-adapted} if there exists a complete set $R_B$ of inequivalent irreducible matrix representations of $B$ such that for all $\rho\in R$, $\rho\downarrow_B=\bigoplus \gamma_s$, for (not neccessarily distinct) representations $\gamma_s$ in $R_B$. The set $R$ is \textbf{adapted to the chain} $A=A_n> A_{n-1}>\cdots> A_0$ if for each $1\leq i\leq n$ there is a complete set $R_i$ of inequivalent representations of $A_i$ such that $R_i$ is $A_{i-1}$-adapted and $R_n=R$. A set of bases for the representation spaces that give rise to adapted representations is an \textbf{adapted basis}.
\end{definition}
For the FFT results of this paper we assume the ability to construct adapted sets of representations. This requirement is not a limitation, as any set of representations is equivalent to an adapted set of representations. One such construction is outlined in \cite{sovi}.

\begin{definition} A \textbf{quiver} $Q$ is a directed multigraph with vertex set $V(Q)$ and edge set $E(Q)$. For an arrow (directed edge) $e\in E(Q)$ from vertex $\beta$ to vertex $\alpha$, we call $\alpha$ the \textbf{target}, $t(e)$, of $e$ and $\beta$ the \textbf{source}, $s(e)$, of $e$. 
A quiver $Q$ is \textbf{graded} if there is a function $gr:V(Q)\rightarrow \mathbb{N}$ such that for each $e\in E(Q)$, $gr(t(e))>gr(s(e))$.  
\end{definition}

For $A_i$ a subalgebra of $A_{i+1}$ consider a chain of semisimple algebras $A_n > A_{n-1} > \dots > A_1> A_0 $. To associate a graded quiver to this chain, we follow the language of \cite{seminormal}. Let $\rho$ be an irreducible representation of $A_{i}$, i.e., an irreducible $A_i$-module.  Upon restriction to $A_{i-1}$,  $\rho\downarrow_{A_{i-1}}$ decomposes as a direct sum of irreducible $A_{i-1}$-modules. For $\gamma$ an irreducible representation of $A_{i-1}$, let $M(\rho,\gamma)$ denote the multiplicity of $\gamma$ in $\rho\downarrow_{A_{i-1}}$. 
\begin{definition}\label{brattdef} For a chain of semisimple algebras $A_n > A_{n-1} > \dots > A_1> A_0 $ the \textbf{associated Bratteli diagram} is the graded quiver described by
\begin{itemize}
\item[(i)] The vertices of grading $i$ are labeled by the (equivalence classes of) irreducible representations of $A_i$;

\item[(ii)] A vertex labeled by an irreducible representation $\gamma$ of $A_{i-1}$ is connected to a vertex labeled by an irreducible representation $\rho$ of $A_{i}$ by $M(\rho,\gamma)$ arrows. 
\end{itemize} 

\end{definition}
For a Bratteli diagram $\mathcal{B}$, let $\mathcal{B}^i$ denote the set of vertices of grading $i$ in $\mathcal{B}$.

\begin{example}[Brauer algebras.] Brauer algebras are among the non-group algebras of interest in this paper. We denote the Brauer algebra on $n$ points as $\mathcal{B}r_n>$. (See Section \ref{Brauerback} for a brief description of the Brauer algebra.) 

Irreducible representations of $\mathcal{B}r_i$ are indexed by partitions of $i-2k$, $0\leq k\leq i/2$, with a ``branching rule" like that of the symmetric group where an edge between $\rho\in\mathcal{B}^i$ and $\lambda\in\mathcal{B}^{i-1}$ if $\rho$ is obtained from $\lambda$ by adding or removing a box \cite{leduc-ram}. Figure \ref{BrnBratt} shows the Bratteli diagram for the chain of Brauer algebras $\mathcal{B}r_3>\mathcal{B}r_2>\mathcal{B}r_1>\mathcal{B}r_0$. The grading of the Bratteli diagram is listed at the top. (We distinguish $\mathcal{B}r_1$ from $\mathcal{B}r_0$ for convenience in future indexing so that vertices at level $i$ correspond to representations of $\mathcal{B}r_i$.) The Brauer algebra Bratteli diagram of Figure \ref{BrnBratt} is an example of a \textbf{multiplicity-free} diagram in that there is at most one edge from any vertex of grading $i$ to any vertex of grading $i+1$.
\end{example}

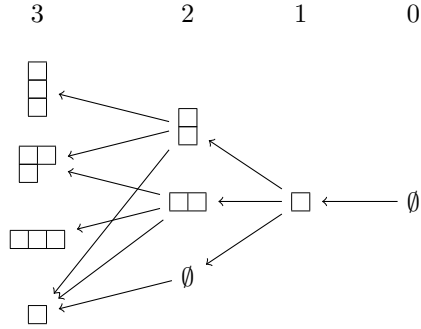
\begin{figure}[ht]\begin{center}
\begin{tikzpicture}[shorten >=1pt,node distance=2cm,on grid,auto,/tikz/initial text=] 
   \node at (-6,2.5) (h) {$\tiny{3}$};
   \node at (-4,2.5) (h) {$\tiny{2}$};
   \node at (-2.5,2.5) (h) {$\tiny{1}$};
   \node at (-1,2.5) (h) {$\tiny{0}$};

   \node at (-6,1.5) (34) {$\tiny{\yng(1,1,1)}$};
   \node at (-6,0.5) (33) {$\tiny{\yng(2,1)}$};
   \node at (-6,-.5) (32) {$\tiny{\yng(3)}$};
   \node at (-6,-1.5) (31) {$\tiny{\yng(1)}$};
   \node at (-4,1) (23) {$\tiny{\yng(1,1)}$};
   \node at (-4,0) (22) {$\tiny{\yng(2)}$};
   \node at (-4,-1) (21) {$\emptyset$};
   \node at (-2.5,0) (1) {$\tiny{\yng(1)}$};
   \node at (-1,0) (0) {$\emptyset$};
   \path[every node/.style={font=\scriptsize}, ->]
    (0) edge  node {} (1)
    (1) edge node {} (21)
    (1) edge node {} (23)
    (1) edge node {} (22)
    (21) edge node {} (31)
    (22) edge node {} (31)
    (22) edge node {} (32)
    (22) edge node {} (33)
    (23) edge node {} (31)
    (23) edge node {} (33)
    (23) edge node {} (34);
    
\end{tikzpicture}
\caption{Bratteli diagram for $\mathcal{B}r_3>\mathcal{B}r_2>\mathcal{B}r_1>\mathcal{B}r_0$}
\label{BrnBratt}
\end{center}\end{figure}
Given a Bratteli diagram $\mathcal{B}$, there is a canonical chain of algebras associated to $\mathcal{B}$ called the {\em chain of path algebras}. For more details, see e.g. \cite{towers, braidgroups}. 

\begin{definition} Let $\mathcal{B}$ be a Bratteli diagram. The \textbf{path algebra (at level i)}, denoted $\mathbb{C}[\mathcal{B}_i]$, is the $\mathbb{C}$-vector space with basis given by ordered pairs of paths  of length $i$ in $\mathcal{B}$ which start at the root and end at the same vertex at level $i$. 
\end{definition}

Note that for a vertex $v$, labeled by a representation $\rho$, the dimension of $\rho$ is given by the number of paths from the root to $v$. Moreover, each path corresponds to a subgroup-equivariant embedding of $\mathbb{C}$ into the representation space of $\rho$ (for more details, see \cite{towers, sovII}).

Further, $\mathbb{C}[\mathcal{B}_i]$ embeds into $\mathbb{C}[\mathcal{B}_{i+1}]$ as a subalgebra by mapping any pair of paths $(P,Q)\in\mathbb{C}[\mathcal{B}_i]$ to the sum $$\sum_e (e\circ P, e\circ Q),$$
over all arrows $e$ such that the source of $e$ is the target of $P$ (equivalently, of $Q$), and $\circ$ denotes concatenation of paths. Thus, elements in these subalgebras are effectively determined by the initial ``legs" of their paths. This is also equivalent to a choice of basis  in the corresponding Wedderburn decomposition of the algebra as a direct sum of matrix algebras, recognizing that for a given element, a number of irreducible matrix elements will take on the same value (equal to the total number of distinct paths that have the common middle ``source" the target of $P$). Identification of this kind of common ``unit" (formalized by the injection of one quiver into another) is the fundamental observation and technique of the quiver-based SOV approach. 

Multiplication in the path algebra $\mathbb{C}[\mathcal{B}_i]$ is the linear extension of $(P,Q)*(P',Q')=\delta_{QP'}(P,Q')$
and is illustrated in Figure \ref{multinpath}. The first arrow represents gluing two pairs of paths along identical middle paths $Q=P'$ and the second arrow represents summation over all possible gluings.
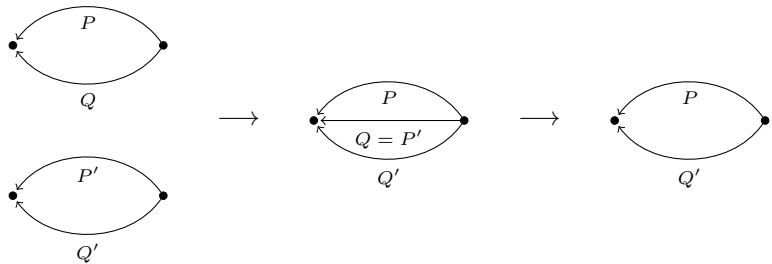
\begin{figure}[H]\begin{center}
\begin{tikzpicture}[shorten >=1pt,node distance=2cm,on grid,auto,/tikz/initial text=] 
   \node[pnt] at (-3,0) (q_0) {};
   \node[pnt] (q_2) [left of=q_0] {};
   \node[pnt] (q_1) [below of=q_0] {};
   \node[pnt] (q_3) [below of=q_2] {};
   \node at (-2,-1)(y1){$\longrightarrow$};
   \node[pnt] at (-1,-1) (y4) {};
   \node[pnt] (y5) [right of=y4] {};
   \node at (2,-1)(y6){$\longrightarrow$};
   \node[pnt] at (3,-1) (y8) {};
   \node[pnt] (y9) [right of=y8] {};
   \path[->,every node/.style={font=\scriptsize}]
    (q_0) edge [bend left=55] node {$Q$} (q_2)
    (q_0) edge [bend right=55] node {$P$} (q_2)
    (q_1) edge [bend left=55] node {$Q'$} (q_3)
    (q_1) edge [bend right=55] node {$P'$} (q_3)
    (y5) edge [bend left=55] node {$Q'$} (y4)
    (y5) edge  node {$Q=P'$} (y4)
    (y5) edge [bend right=55] node {$P$} (y4)
    (y9) edge [bend left=55] node {$Q'$} (y8)
    (y9) edge  [bend right=55] node {$P$} (y8);    
\end{tikzpicture}
\caption{Multiplication in the path algebra.}
\label{multinpath}
\end{center}
\end{figure}

For a Bratteli diagram $\mathcal{B}$ associated to a chain of semisimple algebras $A_n>A_{n-1}>\cdots >A_0$, consider the associated chain of path algebras:
$\mathbb{C}[\mathcal{B}_n]>\mathbb{C}[\mathcal{B}_{n-1}]>\cdots> \mathbb{C}[\mathcal{B}_1]>\mathbb{C}[\mathcal{B}_0].$
It is not too difficult to see that there exists an isomorphism between these algebra chains.
\begin{lemma}\label{bratchain}Let $A=A_n > A_{n-1} > \dots > A_1> A_0$ be a chain of semisimple algebras with Bratteli diagram $\mathcal{B}$. Then the chain of path algebras associated to $\mathcal{B}$ is isomorphic to the group algebra chain.
\end{lemma}
Lemma \ref{bratchain} the key translation in the group algebra case \cite{sovII} as it allows for computation of the Fourier transform to be reformulated in the path algebra. As this result holds in the semisimple algebra setting, we can extend the SOV approach to any semisimple algebra.

\subsection{Gel'fand-Tsetlin bases} The analogous concept in the path algebra of adapted bases associated to a group algebra chain is a \textit{system of Gel'fand-Tsetlin bases}.

\begin{definition}\label{defgel} Let $\mathcal{B}$ be the Bratteli diagram associated to a chain of group algebras. A \textbf{system of Gel'fand-Tsetlin bases for} $\mathbf{\mathcal{B}}$ consists of a collection of bases for the representation spaces $\{V_\alpha\vert\;\alpha\in V(\mathcal{B})\}$ of the representations corresponding to $\alpha$ indexed by paths from the root to $\alpha$, along with maps from the paths to the basis vectors; i.e., a set of basis vectors along with knowledge of the path corresponding to each vector.
\end{definition}

Gel'fand-Tsetlin bases provide a means to better understand the isomorphism of Lemma \ref{bratchain} between a chain of semisimple algebras and the corresponding chain of path algebras. Since Gel'fand-Tsetlin bases are indexed by paths in $\mathcal{B}$ and a basis for the path algebra $\mathbb{C}[\mathcal{B}_n]$ consists of pairs of paths, we identify the semisimple algebra $A$ with its realization in coordinates relative to the Gel'fand-Tsetlin basis, indexed by pairs of paths of length $n$ in $\mathcal{B}$ that share the same endpoint. For a complete set $R$ of inequivalent irreducible representations of $A$ adapted to the chain $A=A_n>A_{n-1}>\dots>A_0$ let $D_i$ be the dual matrix coefficient basis for $A_i$ associated to $R$. For $a\in A$ let $\tilde{a}$ be the image of $a$ in the path algebra under the path algebra isomorphism. Note that the image of the set of dual matrix coefficient bases $D_i$ is a Gel'fand Tsetlin basis. Then Lemma \ref{equivcomp} becomes
\begin{lemma}\label{equivcomp2}  The computation of the Fourier transform of $f$ (originally expressed with respect to a basis $\{a_i\}_{i\in I}$) on a semisimple algebra $A$ with respect to a complete set of inequivalent irreducible representations $R$ adapted to the chain $A_n>A_{n-1}>\dots >A_0$ is the same as computation (rewriting) of $$\sum_{i\in I}f(a_i)\tilde{a}_i,$$  relative to a Gel'fand Tsetlin basis for the path algebra associated to the chain.
\end{lemma}

\section{The Separation of Variables Approach}\label{sepvarstate}

At the heart of the SOV approach are two main steps. The first is to express a path algebra element as a factorization over subsets of the Bratteli diagram in such a way as to disentangle the dependencies in the sum. To extend to the semisimple algebra setting, we need a coset-like factorization of basis elements of $A$.
\begin{definition} Let $A$ be a semisimple algebra with basis $\hat{A}$ and $B$ a subalgebra with basis $\hat{B}$. A \textbf{factor set} for $A$ over $B$ is a set $Y\subseteq \hat{A}$ such that each basis element of $A$ can be written as $yb$, with $y\in Y$ and $b\in\hat{B}$. 
\end{definition}

Note that a factor set is weaker than the group notion of a set of coset representatives. However, since $\hat{A}\subseteq\{yb\mid y\in Y, b\in \hat{B}\}$, computation of $\sum f(yb)yb$ requires more operations (or the same number) as computation of $\sum f(a_i)a_i$, so we may use it to bound computation of the Fourier transform.  

For $Y$ a factor set of $A$ over $B$, $\tilde{Y}=\{\tilde{y}\mid y\in Y\}$, and $F_y$ (for each $y\in Y$) an arbitrary element in the path algebra of $B$, define  
$$m_A(R,Y,B)=\frac{1}{\dim(A)}\times
    \left\{
     \begin{array}{llll}
        \text{minimum number of operations required to compute }\\
    \sum_{y\in \tilde{Y}} yF_y\text{ in a system of Gel'fand-Tsetlin bases for } \mathcal{B}.\\
     \end{array} 
   \right. $$

\begin{lemma}\label{factorsum} 

Let $B$ be a subalgebra of $A$, $R$ a complete $B$-adapted set of inequivalent irreducible matrix representations of $A$,  and $Y$ a factor set for $A$ over $B$. Then
$$t_A(R)\leq t_B(R_{B})+m_A(R,Y,B).$$
\end{lemma}

Lemma \ref{factorsum} is a restatement of Lemma 3.1 of \cite{sovII}, Lemma 2.10 of \cite{maslen} and Proposition 1 of \cite{diarock}.  It shows that  to compute the Fourier transform of a complex function defined on $A$ at a set of $B$-adapted representations, we need only compute   the pieces
$$\mathcal{F}_Y:=\sum_{y\in \tilde{Y}}yF_y.$$

In doing so, the complexity estimate ``reduces" to a close study of the computation of $\mathcal{F}_Y$. This idea can  be iterated through a chain of subalgebras.  Assuming a set of representations $R$ adapted to a chain $A=A_n> A_{n-1}>\cdots> A_0$ and factor sets $Y_i\subseteq A_i$, iteration  of Lemma \ref{factorsum} gives \begin{equation}\label{iterate}t_A(R)\leq t_{A_0}(R_{A_0})+\sum_{i=1}^n m_{A_i}(R_{A_i},Y_i, A_{i-1}).\end{equation}

Lemma \ref{equivcomp2} casts computation of a Fourier transform on $A$ in terms of computation in the path algebra. 
Lemma \ref{factorsum} shows how this can be accomplished  via factoring in the semisimple algebra. In particular, we rely on the special and sparse structure of elements that are in the intersection of sublalgebras and centrailizers of subalgebras, a class of elements of particular importance in the examples of interest for this paper. This follows the approach taken in \cite{sovII}  wherein further details and examples can be found. We outline the ideas below.


\begin{definition}\label{widef} Let $\mathcal{B}$ be a Bratteli diagram with highest grading at least $n$ corresponding to a subalgebra chain for $A$. For $X\subseteq (\mathbb{C}[\mathcal{B}_n])^m=\mathbb{C}[\mathcal{B}_n]\times\cdots\times\mathbb{C}[\mathcal{B}_n]$,
let $i^+$ denote the smallest integer such that $x_i\in \mathbb{C}[\mathcal{B}_{i^+}]$ for all $i$th entries $x_i$ of elements of $X$. Similarly, let $i^-$ denote the largest integer less than or equal to $i^+$ such that $x_i\in\cent(\mathbb{C}[\mathcal{B}_{i^-}])$ for all $i$th entries $x_i$ of elements of $X$. Then for $1\leq i\leq m$ define $$X_i:=\mathbb{C}[\mathcal{B}_{i^+}]\cap\cent(\mathbb{C}[\mathcal{B}_{i^-}]).$$ \end{definition}

To each space $X_i$, associate the quiver $Q_i$ of Figure \ref{Qi}. (Note that $Q_i$ is also the quiver associated to every element of $X_i$.) Let $\Hom(Q_i;\mathcal{B})$ denote the set of morphisms of $Q_i$ into the Bratteli diagram $\mathcal{B}$ and  $A(Q_i;\mathcal{B})$ denote the space of finitely supported formal $\mathbb{C}$-linear combinations of such morphisms. By Lemma 5.5 of \cite{sovII}, $X_i\cong A(Q_i;\mathcal{B})$. Thus, $\dim (X_i)=\#\Hom(Q_i;\mathcal{B})$, for $\#\Hom(Q_i;\mathcal{B})$ the number of morphisms from $Q_i$ into $\mathcal{B}$. 
 
\begin{figure}[H]\begin{center}
\begin{tikzpicture}[shorten >=1pt,node distance=2cm,on grid,auto,/tikz/initial text=] 
  \node at (-4.25,-1) (q) {$Q_i$};
   \node[pnt] at (-2,0) (02) {};
   \node[pnt] at (-6.5,0) (n2) {};
   \node[pnt] at (-3.5,0) (11) {};
   \node at (-3.25,-.2) (l) {\footnotesize$i^-$};
   \node[pnt] at (-5,0) (31) {};
   \node at (-5.3,-.2) (l) {\footnotesize$i^+$};
   \draw (02) node[below] {\footnotesize $0$};
   \draw (n2) node[below] {\footnotesize $n$};
   \path[->,every node/.style={font=\scriptsize}]
    (02) edge node {} (11)
    (31) edge node {} (n2)
    (11) edge [bend left=55] node {} (31)
    (11) edge [bend right=55] node {} (31);  

\end{tikzpicture}
\caption{The quiver associated to $X_i$ and $x_i$.}
\label{Qi}
\end{center}\end{figure}
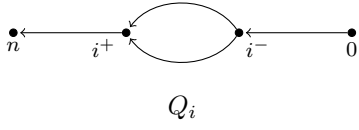

In this setting (bilinear) group algebra multiplication is transformed into a bilinear map on products of associated spaces of quiver morphisms $*:A(Q_1;B)\times A(Q_2;B)\rightarrow A(Q_1\triangle Q_2; B)$, for $Q_1\triangle Q_2$ the symmetric difference of $Q_1$ and $Q_2$, i.e., the induced graph on the edges of $Q_1\cup Q_2$ not in $Q_1\cap Q_2$. 
\begin{definition}\label{symdif} Let $R$ be a graded quiver with subquivers $Q_1, Q_2$. Let $E_1$ (respectively $E_2$) be the edge set of $Q_1$ (respectively $Q_2$). The \textbf{symmetric difference}, $Q_1\triangle Q_2$, of $Q_1$ and $Q_2$ is the induced graph on the edge set $(E_1\cup E_2)\setminus (E_1\cap E_2).$ 
\end{definition}

We now define separation of variables for the computation of a Fourier transform on an algebra: 

\begin{palgorithm}\label{alg2}
\begin{itemize}
\item[]
\item[I.] Choose $m\in\mathbb{N}$ and a subset $X\subseteq (\mathbb{C}[\mathcal{B}_n])^m$ such that $\vert X\vert=\vert Y\vert$ and for each $y\in \tilde{Y}$ there exists $(x_1,\dots, x_m)\in X$ with $yF_y=x_1\cdots x_m.$ Thus, $X$ can be thought of as a choice of factorization into   $m$ elements (some of which may be the identity) of each term $yF_y$. 
\item[II.] For $1\leq i\leq m$ let $X_i$ be as in Definition \ref{widef}. For $\sigma\in S_m$, let $w_i= x_{\sigma(i)}$. The   bilinear map $*$ 
is such that 
 $x_1\cdots x_m=(((w_1*w_2)*w_3)\cdots*w_m),$

\item[III.] For $  0\leq i<m$, let $W_i=\{(w_{i+1},\dots, w_m)\mid (x_1,\dots,x_m)\in X\}$. Let $
W_m=\emptyset. $ Note that $W_i\subseteq X_{\sigma(i+1)}\times\cdots\times X_{\sigma(m)}$.
\end{itemize}
\end{palgorithm}
The SOV approach gives a method for organizing computation $\sum_{y\in \tilde{Y}}yF_y$ in a manner that allows the complexity to be determined by counting the number of occurrences of subgraphs in the Bratteli diagram. Proof of the following theorem follows Theorem 3.8 from \cite{sovII}, essentially word for word, but in the more general setting of semisimple algebras. 
\begin{theorem}\label{efficiencyfirststate}
For $x_i$ and $\sigma$ as above, let $Q_i^\sigma$ denote the quiver associated to $w_i=x_{\sigma(i)}$. Then we may compute $\sum_{y\in \tilde{Y}} yF_y$ in at most $$\sum_{i=1}^{m-1}\vert W_{i-1}\vert\#\Hom((Q_1^\sigma\triangle \cdots \triangle Q_i^\sigma)\cup Q_{i+1}^\sigma;\mathcal{B})$$ multiplications and fewer additions. 
\end{theorem}

\section{The complexity of Fourier transforms on the Brauer and BMW algebras}\label{mainresults}

The SOV approach first factors the elements of a factor set, then translates path algebra multiplication into maps indexed by subgraphs. The complexity is determined by the size of the factorization sets and the number of occurrences of these subgraphs in the Bratteli diagram. In this section we apply these ideas to the Braer and BMW algebras to give complexity results for Fourier transforms on these algebras.

For a parameter $q$, the Brauer algebra is a semisimple $\mathbb{C}(q)$-algebra, while the BMW algebra is a ``deformation" of the Brauer algebra. For $q=1$, the group algebra of the symmetric group is a subalgebra of the Brauer algebra: $\mathbb{C}[S_n]< \mathcal{B}r_n$.  As such, these results are natural extensions of Fourier transforms of functions on the symmetric group, and in fact the proof of Theorem \ref{brauerthm} yields the same diagrams as in \cite{maslen}.

\subsection{Background: The Brauer Algebra}\label{Brauerback}
An element in the symmetric group $S_n$ is realized as a diagram on $2n$ points, consisting of two rows of $n$ points each, with each point in the top row connected by an edge to exactly one point in the bottom row (see Figure \ref{sym1}). For two elements $x,y$ in $S_n$, the product $xy$ is the concatenation of the two diagrams: to compute the product $xy$, place the diagram for $x$ on top of the one for $y$ and trace the edges from top to bottom (note that we consider multiplication from left to right).
\begin{figure}[H] 
\begin{center}
\begin{tikzpicture}
   \node[pnt] at (-3,0) (v_1) {};
   \node[pnt] (v_2) [right of=v_1] {};
   \node[pnt] (v_3) [right of=v_2] {};
   \node[pnt] (v_4) [right of=v_3] {};
   \node[pnt] (v_5) [below of=v_1] {};
   \node[pnt] (v_6) [below of=v_2] {};
   \node[pnt] (v_7) [below of=v_3] {};
   \node[pnt] (v_8) [below of=v_4] {};
   \draw (v_1) node[above] {\footnotesize $1$};
   \draw (v_2) node[above] {\footnotesize $2$};
   \draw (v_3) node[above] {\footnotesize $3$};
   \draw (v_4) node[above] {\footnotesize $4$};
   \draw (v_5) node[below] {\footnotesize $1$};
   \draw (v_6) node[below] {\footnotesize $2$};
   \draw (v_7) node[below] {\footnotesize $3$};
   \draw (v_8) node[below] {\footnotesize $4$};

   \path[-]
    (v_1) edge node {} (v_7)
    (v_2) edge  node {} (v_8)
    (v_3) edge  node {} (v_6)
    (v_4) edge node {} (v_5);
      \end{tikzpicture}
\caption{(1324)}
\label{sym1}
\end{center}
\end{figure}

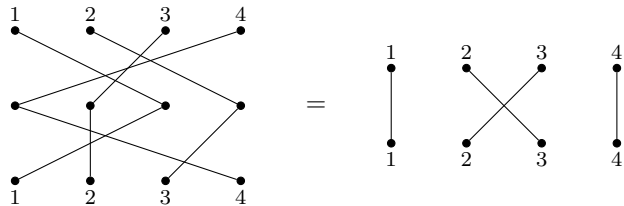
\begin{figure}[H] 
\begin{center}
\begin{tikzpicture}
   \node[pnt] at (-3,0) (v_1) {};
   \node[pnt] (v_2) [right of=v_1] {};
   \node[pnt] (v_3) [right of=v_2] {};
   \node[pnt] (v_4) [right of=v_3] {};
   \node[pnt] (v_5) [below of=v_1] {};
   \node[pnt] (v_6) [below of=v_2] {};
   \node[pnt] (v_7) [below of=v_3] {};
   \node[pnt] (v_8) [below of=v_4] {};
   \node[pnt] (w_5) [below of=v_5] {};
   \node[pnt] (w_6) [below of=v_6] {};
   \node[pnt] (w_7) [below of=v_7] {};
   \node[pnt] (w_8) [below of=v_8] {};
   \node (q) [right of=v_8] {=};
   \node[pnt] at (2,-.5) (u_1) {};
   \node[pnt] (u_2) [right of=u_1] {};
   \node[pnt] (u_3) [right of=u_2] {};
   \node[pnt] (u_4) [right of=u_3] {};
   \node[pnt] (u_5) [below of=u_1] {};
   \node[pnt] (u_6) [below of=u_2] {};
   \node[pnt] (u_7) [below of=u_3] {};
   \node[pnt] (u_8) [below of=u_4] {};

   \draw (v_1) node[above] {\footnotesize $1$};
   \draw (v_2) node[above] {\footnotesize $2$};
   \draw (v_3) node[above] {\footnotesize $3$};
   \draw (v_4) node[above] {\footnotesize $4$};
   \draw (w_5) node[below] {\footnotesize $1$};
   \draw (w_6) node[below] {\footnotesize $2$};
   \draw (w_7) node[below] {\footnotesize $3$};
   \draw (w_8) node[below] {\footnotesize $4$};

   \draw (u_1) node[above] {\footnotesize $1$};
   \draw (u_2) node[above] {\footnotesize $2$};
   \draw (u_3) node[above] {\footnotesize $3$};
   \draw (u_4) node[above] {\footnotesize $4$};
   \draw (u_5) node[below] {\footnotesize $1$};
   \draw (u_6) node[below] {\footnotesize $2$};
   \draw (u_7) node[below] {\footnotesize $3$};
   \draw (u_8) node[below] {\footnotesize $4$};

     \path[-]
    (v_1) edge node {} (v_7)
    (v_2) edge  node {} (v_8)
    (v_3) edge  node {} (v_6)
    (v_4) edge node {} (v_5)
    (v_5) edge node {} (w_8)
    (v_6) edge  node {} (w_6)
    (v_7) edge  node {} (w_5)
    (v_8) edge node {} (w_7)   
    (u_1) edge node {} (u_5)
    (u_2) edge  node {} (u_7)
    (u_3) edge  node {} (u_6)
    (u_4) edge node {} (u_8);    

\end{tikzpicture}
\caption{(1324)*(143)=(23)}\label{concat}
\end{center}
\end{figure}

The simple transpositions $\{r_i=(i\;i+1)\mid 1\leq i\leq n-1\}$ form a generating set for the symmetric group.

Elements of the Brauer monoid, $Br_n$, are realized by generalizing symmetric group diagrams: consider diagrams on $2$ rows of $n$ points each, with edges connecting pairs of points regardless of row and each point part of exactly one edge. Multiplication is again concatenation of diagrams. Note that in some cases, this introduces a closed loop. A parameter $q$ is used to keep track of the number of closed loops: for two diagrams $x,y\in Br_n$, let $c$ denote the number of closed loops in the multiplication $xy$ and let $z$ be the diagram of this product with the closed loops removed. Then $xy=q^c z$.  

\begin{figure}[H] 
\begin{center}
\begin{tikzpicture}
   \node[pnt] at (-3,0) (v_1) {};
   \node[pnt] (v_2) [right of=v_1] {};
   \node[pnt] (v_3) [right of=v_2] {};
   \node[pnt] (v_4) [right of=v_3] {};
   \node[pnt] (v_5) [below of=v_1] {};
   \node[pnt] (v_6) [below of=v_2] {};
   \node[pnt] (v_7) [below of=v_3] {};
   \node[pnt] (v_8) [below of=v_4] {};
   \node[pnt] (w_5) [below of=v_5] {};
   \node[pnt] (w_6) [below of=v_6] {};
   \node[pnt] (w_7) [below of=v_7] {};
   \node[pnt] (w_8) [below of=v_8] {};
   \node (q) [right of=v_8] {=};
   \node (x) at (-3.5,-.5) {$x$};
   \node (x) at (-3.5,-1.5) {$y$};
   \node (q) at (1.5,-1) {$q$};
   \node[pnt] at (2,-.5) (u_1) {};
   \node[pnt] (u_2) [right of=u_1] {};
   \node[pnt] (u_3) [right of=u_2] {};
   \node[pnt] (u_4) [right of=u_3] {};
   \node[pnt] (u_5) [below of=u_1] {};
   \node[pnt] (u_6) [below of=u_2] {};
   \node[pnt] (u_7) [below of=u_3] {};
   \node[pnt] (u_8) [below of=u_4] {};
   \node (x) at (5.5,-1) {$z$};
     \path[-]
    (v_1) edge [bend right=25] node {} (v_4)
    (v_2) edge  node {} (v_8)
    (v_3) edge  node {} (v_6)
    (v_7) edge [bend right=25] node {} (v_5)
    (v_7) edge [bend left=25] node {} (v_5)
    (v_6) edge node {} (w_8)
    (w_6) edge [bend right=25]   node {} (w_5)
    (v_8) edge node {} (w_7)   
    (u_1) edge [bend right=25]  node {} (u_4)
    (u_2) edge  node {} (u_7)
    (u_3) edge  node {} (u_8)
    (u_5) edge [bend left=25]  node {} (u_6);    

\end{tikzpicture}
\caption{ $xy=q^1z$}\label{brauer}
\end{center}
\end{figure}
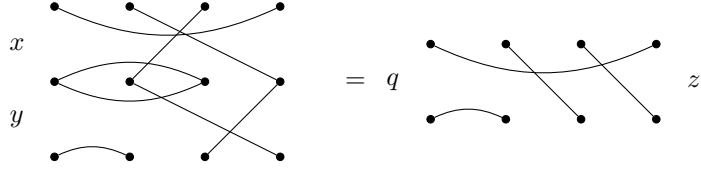

Two Brauer diagrams $d_1$ and $d_2$ are \textit{equivalent} if they differ only in the number of closed loops, i.e., if when $q=1$, $d_1=d_2$. For example, for $x,y,z$ as in Figure \ref{brauer}, the product $xy$ is equivalent to $z$. The Brauer monoid, $Br_n$ consists of the set of equivalence classes of such diagrams and is generated by the set of elements $\{r_i, e_i \mid 1\leq i\leq n-1\}$ (see Figure \ref{figg}). Note that the symmetric group $S_n$ is generated by the transpositions $\{r_i\mid 1\leq i \leq n-1\}$ and so $S_n\subseteq Br_n$. 

\begin{figure}[H] 
\begin{center}
\begin{tikzpicture}
   \node[pnt] at (-3,0) (v_1) {};
   \node[pnt] (v_2) [right of=v_1] {};
   \node[pnt] (v_3) [right of=v_2] {};
   \node[pnt] (v_4) [right of=v_3] {};
   \node[pnt] (v_5) [right of=v_4] {};
   \node[pnt] (v_6) [right of=v_5] {};
   \node at (-2.5,-.5) {$\dots$};
   \node at (1.5,-.5) {$\dots$};
   \node at (-1,.5) {\small$i$};
   \node at (0,.5) {\small$i+1$};
   \node at (-.5,-2) {\large${r_i}$};
   \node[pnt] (v_7) [below of=v_1] {};
   \node[pnt] (v_8) [below of=v_2] {};
   \node[pnt] (v_9) [below of=v_3] {};
   \node[pnt] (v_10) [below of=v_4] {};
   \node[pnt] (v_11) [below of=v_5] {};
   \node[pnt] (v_12) [below of=v_6] {};
   \node[pnt] at (4,0) (w_1) {};
   \node[pnt] (w_2) [right of=w_1] {};
   \node[pnt] (w_3) [right of=w_2] {};
   \node[pnt] (w_4) [right of=w_3] {};
   \node at (4.5,-.5) {$\dots$};
   \node at (8.5,-.5) {$\dots$};
   \node at (6,.5) {\small$i$};
   \node at (7,.5) {\small$i+1$};
   \node at (6.5,-2) {\large${e_i}$};
   \node[pnt] (w_5) [right of=w_4] {};
   \node[pnt] (w_6) [right of=w_5] {};
   \node[pnt] (w_7) [below of=w_1] {};
   \node[pnt] (w_8) [below of=w_2] {};
   \node[pnt] (w_9) [below of=w_3] {};
   \node[pnt] (w_10) [below of=w_4] {};
   \node[pnt] (w_11) [below of=w_5] {};
   \node[pnt] (w_12) [below of=w_6] {};
   \path[-]
    (v_1) edge node {} (v_7)
    (v_2) edge  node {} (v_8)
    (v_3) edge  node {} (v_10)
    (v_5) edge  node {} (v_11)
    (v_6) edge node {} (v_12)
    (v_4) edge node {} (v_9)
    (w_1) edge node {} (w_7)
    (w_2) edge  node {} (w_8)
    (w_3) edge [bend right=50] node {} (w_4)
    (w_10) edge [bend right=50]  node {} (w_9)
    (w_5) edge  node {} (w_11)
    (w_6) edge node {} (w_12);
\end{tikzpicture}
\caption{$r_i, e_i\in Br_n$}\label{figg}
\end{center}
\end{figure}
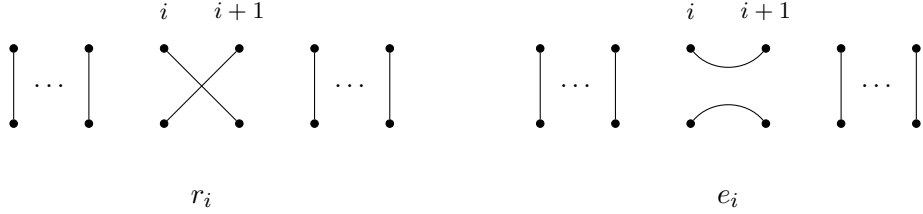

The Brauer algebra, $\mathcal{B}r_n$, is the $\mathbb{C}(q)$-algebra with basis $Br_n$ and dimension $(2n-1)!!$.  Equivalently (see, e.g., \cite{BenkhartShaderRam}), $\mathcal{B}r_n$ has algebraic presentation given by generating set $$\{r_i, e_i \mid 1\leq i\leq n-1\},$$ and relations:

$$\begin{array}{llll}
(1)& r_i^2=1, &(2)& r_ir_j=r_jr_i, \;\;\;\;r_ie_j=e_jr_i,\;\;\;\; e_ie_j=e_je_i,\;\; |i-j|>1\\
 (3)& e_i^2=qe_i, &(4)& e_ir_i=r_ie_i=e_i, \\
 (5)&r_ir_{i+1}r_i=r_{i+1}r_ir_{i+1},& (6)&e_ie_{i+1}e_i=e_i,\;\;\;\; e_{i+1}e_ie_{i+1}=e_{i+1},\\
(7)& r_ie_{i+1}e_i=r_{i+1}e_i, & (8)& e_{i+1}e_ir_{i+1}=e_{i+1}r_i.\\
 \end{array}$$
 
In \cite{wenzl}, Wenzl showed that the Brauer algebra, $\mathcal{B}r_n(q)$, is a semisimple algebra over $\mathbb{C}(q)$. In fact, replacing $q$ by $\alpha\in\mathbb{C}$, $\mathcal{B}r_n(\alpha)$ is semisimple for all but finitely many integers $\alpha$ \cite{rui}.

\subsection{Fourier transforms on $\mathcal{B}r_n$}

 We first find a factor set for $\mathcal{B}r_n$ over $\mathcal{B}r_{n-1}$, viewing each diagram in $Br_{n-1}$ as an element of $Br_n$ by adding a point to the end of the top and bottom rows and connecting these two points with an edge. With Lemma \ref{factorsum}, we then use the SOV approach to compute the Fourier transform of $f=\sum_{d\in Br_n} f(d)d$ in $\mathcal{B}r_n$.

Let $R=\{id, r_1\cdots r_{n-1}, r_2\cdots r_{n-1}, \dots, r_{n-1}\}$ and let $ER=\{r_{j}\cdots r_{i-1}e_{i}\cdots e_{n-1}\mid 1\leq i\leq n-1, 1\leq j\leq i-1\}$
Let $Y=R\cup ER$
\begin{lemma}\label{facset}
$Y$ is a factor set for $\mathcal{B}r_n$ over $\mathcal{B}r_{n-1}$.
\end{lemma}
\begin{proof}
First note that $R\subseteq S_n$ forms a complete set of coset representatives for $S_n/S_{n-1}$ \cite{maslen} and so we need only show that for any $d\in Br_n-S_n$, $d=yd'$, for $y\in Y$, $d'\in Br_{n-1}$. 

Due to the final factor $e_{n-1}$, each element of $ER$ has exactly one horizontal edge in its bottom row, connecting the last two points. Each element of $ER$ also has exactly one horizontal edge in its top row, and each possible such edge corresponds to an element of  $ER$. As an example, see Figure \ref{drawafigure}.

Let $d\in Br_n-S_n$. Then $d$ has at least one horizontal edge, $e$, in its top row. Choose an element, $y$, of $ER$ with edge $e$. This determines an element $d'$ in $Br_{n-1}$ with $d=yd'$. For an example, see Figure \ref{drawanotherfigure}. Note that for this example there are two possible choices for $y$, and (though not always the case) $d'$ is the same for each choice of $y$.

\begin{figure}[H]\begin{center}
\begin{tikzpicture}
\node at (-1,0) {};
\node at (-1,-.5) {};
\node[pnt] at (.25,0) (A) {};
\node[pnt] at (.5,0) (B) {};
\node[pnt] at (.75,0) (C1){};
\node[pnt] at (1.0,0) (D1){};
\node[pnt] at (.25,-.5) (A2) {};
\node[pnt] at (.5,-.5) (B2) {};
\node[pnt] at (.75,-.5) (C2) {};
\node[pnt] at (1.0,-.5) (D2){};

\node at (.575,-1) (q) {$e_1e_2e_3$};

\node[pnt] at (1.75,0) (A3) {};
\node[pnt] at (2,0) (B3) {};
\node[pnt] at (2.25,0) (C3){};
\node[pnt] at (1.75,-.5) (A4) {};
\node[pnt] at (2,-.5) (B4) {};
\node[pnt] at (2.25,-.5) (C4) {};
\node[pnt] at (2.5,0) (D3) {};
\node[pnt] at (2.5,-.5) (D4) {};

\node at (2.125,-1) (q) {$r_1e_2e_3$};

\node[pnt] at (3.25,0) (A5) {};
\node[pnt] at (3.5,0) (B5) {};
\node[pnt] at (3.75,0) (C5){};
\node[pnt] at (4,0) (D5){};
\node[pnt] at (3.25,-.5) (A6) {};
\node[pnt] at (3.5,-.5) (B6) {};
\node[pnt] at (3.75,-.5) (C6) {};
\node[pnt] at (4,-.5) (D6){};

\node at (3.575,-1) (q) {$r_1r_2e_3$};

\node[pnt] at (4.75,0) (A7) {};
\node[pnt] at (5,0) (B7) {};
\node[pnt] at (5.25,0) (C7){};
\node[pnt] at (5.5,0) (D7){};
\node[pnt] at (4.75,-.5) (A8) {};
\node[pnt] at (5.0,-.5) (B8) {};
\node[pnt] at (5.25,-.5) (C8) {};
\node[pnt] at (5.5,-.5) (D8){};

\node at (5.125,-1) (q) {$e_2e_3$};

\node[pnt] at (6.25,0) (A9) {};
\node[pnt] at (6.5,0) (B9) {};
\node[pnt] at (6.75,0) (C9){};
\node[pnt] at (7,0) (D9){};
\node[pnt] at (6.25,-.5) (A10) {};
\node[pnt] at (6.5,-.5) (B10) {};
\node[pnt] at (6.75,-.5) (C10) {};
\node[pnt] at (7,-.5) (D10){};

\node at (6.575,-1) (q) {$r_2e_3$};

\node[pnt] at (7.75,0) (A11) {};
\node[pnt] at (8,0) (B11) {};
\node[pnt] at (8.25,0) (C11){};
\node[pnt] at (8.5,0) (D11){};
\node[pnt] at (7.75,-.5) (A12) {};
\node[pnt] at (8.0,-.5) (B12) {};
\node[pnt] at (8.25,-.5) (C12) {};
\node[pnt] at (8.5,-.5) (D12){};

\node at (8.125,-1) (q) {$e_3$};

        \path[-]
   (A) edge[bend right=55] node {} (B)
   (C1) edge node {} (A2)
   (D1) edge node {} (B2)
   (C2) edge[bend left=55] node {} (D2)
   
   (A3) edge[bend right=55] node {} (C3)
   (B3) edge node {} (A4)
   (D3) edge node {} (B4)
   (C4) edge[bend left=55] node {} (D4)

   (A5) edge[bend right=55] node {} (D5)
   (B5) edge node {} (A6)
   (C5) edge node {} (B6)
   (C6) edge[bend left=55] node {} (D6)

   (B7) edge[bend right=55] node {} (C7)
   (A7) edge node {} (A8)
   (D7) edge node {} (B8)
   (C8) edge[bend left=55] node {} (D8)

   (B9) edge[bend right=55] node {} (D9)
   (A9) edge node {} (A10)
   (C9) edge node {} (B10)
   (C10) edge[bend left=55] node {} (D10)

   (C11) edge[bend right=55] node {} (D11)
   (A11) edge node {} (A12)
   (B11) edge node {} (B12)
   (C12) edge[bend left=55] node {} (D12);
   
\end{tikzpicture}
\caption{$ER$ in $Br_4$}\label{drawafigure}
\end{center}\end{figure}
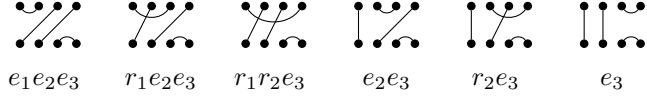
\begin{figure}[H]\begin{center}
\begin{tikzpicture}
\node at (-1,0) {};
\node at (-1,-.5) {};
\node[pnt] at (.25,0) (A) {};
\node[pnt] at (.5,0) (B) {};
\node[pnt] at (.75,0) (C1){};
\node[pnt] at (1.0,0) (D1){};
\node[pnt] at (.25,-.5) (A2) {};
\node[pnt] at (.5,-.5) (B2) {};
\node[pnt] at (.75,-.5) (C2) {};
\node[pnt] at (1.0,-.5) (D2){};

\node at (1.625,-.250) {=};
\node at (-1,-.5) {};
\node[pnt] at (2.25,-0.75) (A5) {};
\node[pnt] at (2.5,-0.75) (B5) {};
\node[pnt] at (2.75,-0.75) (C5){};
\node[pnt] at (3.0,-0.75) (D5){};
\node[pnt] at (2.25,-.25) (A4) {};
\node[pnt] at (2.5,-.25) (B4) {};
\node[pnt] at (2.75,-.25) (C4) {};
\node[pnt] at (3.0,-.25) (D4){};
\node[pnt] at (2.25,.25) (A3) {};
\node[pnt] at (2.5,.25) (B3) {};
\node[pnt] at (2.75,.25) (C3) {};
\node[pnt] at (3.0,.25) (D3){};

\begin{scope}[shift={(2,0)}]
\node at (1.625,-.250) {=};
\node at (-1,-.5) {};
\node[pnt] at (2.25,-0.75) (nA5) {};
\node[pnt] at (2.5,-0.75) (nB5) {};
\node[pnt] at (2.75,-0.75) (nC5){};
\node[pnt] at (3.0,-0.75) (nD5){};
\node[pnt] at (2.25,-.25) (nA4) {};
\node[pnt] at (2.5,-.25) (nB4) {};
\node[pnt] at (2.75,-.25) (nC4) {};
\node[pnt] at (3.0,-.25) (nD4){};
\node[pnt] at (2.25,.25) (nA3) {};
\node[pnt] at (2.5,.25) (nB3) {};
\node[pnt] at (2.75,.25) (nC3) {};
\node[pnt] at (3.0,.25) (nD3){};
\end{scope}

      \path[-]
   (A) edge[bend right=55] node {} (C1)
   (B) edge[bend right=55] node {} (D1)
   (A2) edge[bend left=55] node {} (D2)
   (B2) edge[bend left=55] node {} (C2)
   (C4) edge[bend left=55] node {} (D4)
      
   (A3) edge[bend right=55] node {} (C3)
   (B3) edge node {} (A4)
   (D3) edge node {} (B4)
   (A4) edge[bend right=55] node {} (B4)
   (C4) edge node {} (A5)
   (D4) edge node {} (D5)
   (B5) edge[bend left=55] node {} (C5)

   (nB3) edge[bend right=55] node {} (nD3)
   (nC3) edge node {} (nB4)
   (nA3) edge node {} (nA4)
   (nA4) edge[bend right=55] node {} (nB4)
   (nC4) edge[bend left=55] node {} (nD4)
   (nC4) edge node {} (nA5)
   (nD4) edge node {} (nD5)
   (nB5) edge[bend left=55] node {} (nC5);

\end{tikzpicture}
\caption{$d=r_1e_2e_3d'=r_2e_3d'$}\label{drawanotherfigure}
\end{center}\end{figure}
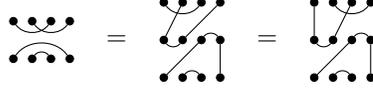
\end{proof}
\begin{theorem}[cf. Theorem \ref{brauerthm}]
The Fourier transform of an element $f$ in the Brauer algebra $\mathcal{B}r_n$ may be computed at a complete set $R$ of irreducible matrix representations of $\mathcal{B}r_n$ adapted to the chain of algebras
$$\mathcal{B}r_n> \mathcal{B}r_{n-1}>\cdots> \mathcal{B}r_0=\mathbb{C}(q)$$
in at most $(4n^2-n+4)\dim(\mathcal{B}r_n)$
operations. 
\end{theorem}
\begin{proof}
By Lemma \ref{facset}, $Y$ is a factor set for $\mathcal{B}r_{n}$ over $\mathcal{B}r_{n-1}$. For  $Y_i=\{id,r_i,e_i\}$, $Y\subseteq\{y_1y_2\cdots y_{n-1}\mid y_i\in Y_i\}$, giving a factorization of $Y$ as required by Step I of the SOV approach.  

Let $\mathcal{B}$ be the Bratteli diagram associated to the chain 
$\mathcal{B}r_n> \mathcal{B}r_{n-1}>\cdots> \mathcal{B}r_0,$
let $\{\mathbb{C}[\mathcal{B}_i]\}$ be the associated chain of path algebras, and let $\tilde{Y}_i=\{\tilde{y}_i\mid y_i\in Y_i\}$. Note that $\tilde{Y}_i\subseteq \mathbb{C}[\mathcal{B}_{i+1}]\cap\cent( \mathbb{C}[\mathcal{B}_{i-1}])$. By Lemma \ref{factorsum}, the complexity of the computation of a Fourier transform of $f$ on $\mathcal{B}r_n$ is bounded by the complexity of computation of:
$$\sum_{\substack{y_i\in \tilde{Y}_i}} y_1\cdots y_{n-1}F_{y_1\cdots y_{n-1}}$$ 
for $F_{y_1\cdots y_{n-1}}\in \mathbb{C}[\mathcal{B}_{n-1}]$. 

We now use the SOV approach.
\begin{itemize}
    \item[I.] Let $X=\{y_1,y_2,\dots,y_{n-1},F_{y_1\cdots y_{n-1}}\mid y_i\in\tilde{Y}_i\}$. 
    \item[II.] Note that $i^+=i+1$ and $i^-=i-1$ for $1\leq i<n$ and that $n^+=n-1$, $n^-=0$. Figure \ref{BnQ} shows the various component subquivers corresponding to the factors $y_i$. They combine together as per Figure \ref{BncalQ} to give the factorization $y_1\cdots y_{n-1}F_{y_1\cdots y_{n-1}}$. Thus, the algorithm proceeds by gluing together quivers $Q_i$ of Figure \ref{BnQ} (corresponding to $\tilde{Y}_i, F_{y_1\cdots y_{n-1}}$) to build the quiver $\mathcal{Q}$ of Figure \ref{BncalQ}. 

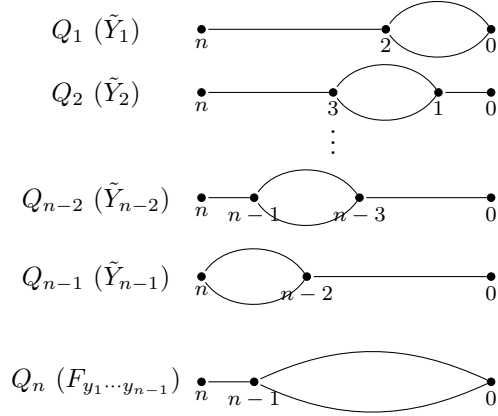
\begin{figure}[H]\begin{center}
\begin{tikzpicture}[shorten >=1pt,node distance=2cm,on grid,auto,/tikz/initial text=,scale=0.7] 
\begin{scope}[shift={(1,0)}]
   \node at (-10,6) (q) {$Q_{n-1}$  $(\tilde{Y}_{n-1})$};
   \node[pnt] at (-2.5,6) (10) {};
   \node[pnt] at (-8,6) (1n) {};
   \node[pnt] at (-6,6) (1n-2) {};
   \node at (-10,7.5) (q) {$Q_{n-2}$ $(\tilde{Y}_{n-2})$};
   \node[pnt] at (-2.5,7.5) (20) {};
   \node[pnt] at (-8,7.5) (2n) {};
   \node[pnt] at (-5,7.5) (2n-3) {};
   \node[pnt] at (-7,7.5) (2n-1) {};
   \node at (-5.5,8.7) (q) {$\vdots$};
\begin{scope}[shift={(0,-.5)}]
   \node at (-10,10) (q) {$Q_2$ ($\tilde{Y}_2$)};
   \node[pnt] at (-2.5,10) (30) {};
   \node[pnt] at (-8,10) (3n) {};
   \node[pnt] at (-3.5,10) (311) {};
   \node[pnt] at (-5.5,10) (333) {};
   \node at (-10,11.2) (q) {$Q_1$ ($\tilde{Y}_1$)};
   \node[pnt] at (-2.5,11.2) (40) {};
   \node[pnt] at (-8,11.2) (4n) {};
   \node[pnt] at (-4.5,11.2) (42) {}; 
  
\end{scope}
\begin{scope}[shift={(-8,11)}] 
  \node at (-2,-7) (q) {$Q_n$ $(F_{y_1\cdots y_{n-1}})$};
   \node[pnt] at (5.5,-7) (011) {};
   \node[pnt] at (0,-7) (n11) {};
   \node[pnt] at (1,-7) (n-11) {};
\end{scope}   
 \end{scope}  
   \draw (10) node[below] {\footnotesize $0$};
   \draw (1n) node[below] {\footnotesize $n$};
   \draw (1n-2) node[below] {\footnotesize $n-2$};
   \draw (20) node[below] {\footnotesize $0$};
   \draw (2n) node[below] {\footnotesize $n$};
   \draw (2n-1) node[below] {\footnotesize $n-1$};
   \draw (2n-3) node[below] {\footnotesize $n-3$};
   \draw (30) node[below] {\footnotesize $0$};
   \draw (3n) node[below] {\footnotesize $n$};
   \draw (333) node[below] {\footnotesize $3$};
   \draw (311) node[below] {\footnotesize $1$};
   \draw (40) node[below] {\footnotesize $0$};
   \draw (4n) node[below] {\footnotesize $n$};
   \draw (42) node[below] {\footnotesize $2$};
   \draw (011) node[below] {\footnotesize $0$};
   \draw (n11) node[below] {\footnotesize $n$};
   \draw (n-11) node[below] {\footnotesize $n-1$};

   \path[every node/.style={font=\scriptsize}]
    (10) edge node {} (1n-2)
    (1n-2) edge [bend left=55] node {} (1n)
    (1n-2) edge [bend right=55] node {} (1n)        
    (20) edge node {} (2n-3)
    (2n-3) edge [bend left=55] node {} (2n-1)
    (2n-3) edge [bend right=55] node {} (2n-1)        
    (2n-1) edge node {} (2n)
    (30) edge node {} (311)
    (311) edge [bend left=55] node {} (333)
    (311) edge [bend right=55] node {} (333)        
    (333) edge node {} (3n)
    (42) edge node {} (4n)
    (40) edge [bend left=55] node {} (42)
    (40) edge [bend right=55] node {} (42)
    (n-11) edge node {} (n11)
    (011) edge [bend left=25] node {} (n-11)
    (011) edge [bend right=25] node {} (n-11);

\end{tikzpicture}
\caption{Component subquivers of the factorization.}
\label{BnQ}
\end{center}\end{figure}

\begin{figure}[H]\begin{center}
\begin{tikzpicture}[shorten >=1pt,node distance=2cm,on grid,auto,/tikz/initial text=] 
   
   \node at (1.25,1.5) (q) {\large $G$};
   \node at (4,.25) (q) {\tiny $\tilde{Y}_2$};
   \node at (3,.25) (q) {\tiny $\tilde{Y}_3$};
   \node at (2,.25) (q) {\tiny $\tilde{Y}_4$};
   \node at (-.5,.25) (q) {\tiny $\tilde{Y}_{n-1}$};
   \node at (-1.5,.25) (q) {\tiny $\tilde{Y}_n$};
   \node at (1.25,-.75) (q) {\tiny $F_{y_2\cdots y_n}$};
   \node[pnt] at (5,0) (01) {};
   \node[pnt] at (4,0) (11) {};
   \node[pnt] at (3,0) (21) {};
   \node[pnt] at (4,.5) (12) {};    
   \node[pnt] at (3,.5) (22) {};
   \node[pnt] at (2,.5) (32) {};
   \node[pnt] at (2,0) (31) {};
   \node[pnt] at (1,.5) (42) {};   
   \node at (.75,.25) (q) {$\dots$};
   \node[pnt] at (.5,0) (n-31) {};
   \node[pnt] at (-.5,.5) (n-22) {};
   \node[pnt] at (-.5,0) (n-21) {};
   \node[pnt] at (-1.5,0) (n-11) {};
   \node[pnt] at (-2.5,.5) (n2) {};
   \node[pnt] at (-1.5,.5) (n-12) {};
   \draw (01) node[below] {\tiny $0$};
   \draw (11) node[below] {\tiny $1$};
   \draw (12) node[above] {\tiny $1$};
   \draw (21) node[below] {\tiny $2$};
   \draw (22) node[above] {\tiny $2$};
   \draw (32) node[above] {\tiny $3$};
   \draw (31) node[below] {\tiny $3$};
   \draw (42) node[above] {\tiny $4$};
   \draw (n-22) node[above] {\tiny $n-2$};
   \draw (n-31) node[below] {\tiny $n-3$};
   \draw (n-12) node[above] {\tiny $n-1$};
   \draw (n-11) node[below] {\tiny $n-1$};
   \draw (n-21) node[below] {\tiny $n-2$};
   \draw (n2) node[above] {\tiny $n$};
   \path[every node/.style={font=\scriptsize}]
    (01) edge node {} (12)
    (01) edge node {} (11)
    (11) edge node {} (21)
    (11) edge node {} (22)
    (21) edge node {} (32)
    (22) edge node {} (32)
    (12) edge node {} (22)
    (32) edge node {} (42)
    (21) edge node {} (31)    
    (31) edge node {} (42)
    (42) edge node {} (n-22)
    (31) edge node {} (n-31)
    (n-31) edge node {} (n-21)
    (n-31) edge node {} (n-22)
    (n-21) edge node {} (n-12)
    (n-21) edge node {} (n-11)
    (n-12) edge node {} (n2)
    (n-22) edge node {} (n-12)    
    (n-11) edge node {} (n2)
    (01) edge [bend left=35] node {} (n-11);

\end{tikzpicture}
\caption{}
\label{BncalQ}
\end{center}\end{figure}
    
Let $\sigma=(n\; n-1\cdots 1)\in S_n$    
    \item[III.] For $\sigma$ as above, $W_i=\{(y_i,\dots,y_{n-1})\mid y_i\in \tilde{Y}_i\}$. Note that $|W_i|=|\tilde{Y}_i||\tilde{Y}_2|\cdots|\tilde{Y}_n|$. Recall that $Q_i^\sigma=Q_{\sigma(i)}$.
    
\end{itemize}

 By Theorem \ref{efficiencyfirststate}, we may compute $\sum_{\substack{y_i\in \tilde{Y}_i}} y_1\cdots y_{n-1}F_{y_1\cdots y_{n-1}}$  (and hence bound computation of $\sum_{y\in\tilde{Y}}yF_y$) in at most 
    $$\sum_{i=1}^{m-1}\vert W_{i-1}\vert\#\Hom((Q_1^\sigma\triangle \cdots \triangle Q_i^\sigma)\cup Q_{i+1}^\sigma;\mathcal{B})$$ multiplications, with $(Q_1^\sigma\triangle \cdots \triangle Q_i^\sigma)\cup Q_{i+1}^\sigma$ as in Figure \ref{GnH1}. Let $\mathcal{H}_i^n$ denote this quiver.

\begin{figure}[H]\begin{center}
\begin{tikzpicture}[shorten >=1pt,node distance=2cm,on grid,auto,/tikz/initial text=] 
   \node at (2.5,-1.5) (q) {\large $\mathcal{H}_i^n$};
   \node at (2.4,.75) (q) {\tiny $\tilde{Y}_i$};
   \node[pnt] at (5,.5) (04) {};
   \node[pnt] at (0,.5) (n4) {};
   \node[pnt] at (3.5,.5) (p5) {};
   \node[pnt] at (2.5,.5) (p+15) {};
   \node[pnt] at (2.5,1) (p+16) {};
   \node[pnt] at (1.5,1) (p+26) {};
   \draw (04) node[below] {\footnotesize $\hat{0}$};
   \draw (n4) node[below] {\footnotesize $\beta_{n-1}$};
   \draw (p5) node[below] {\footnotesize $\beta_{i-2}$};
   \draw (p+15) node[below] {\footnotesize $\beta{i-1}$};
   \draw (p+16) node[above] {\footnotesize $\alpha_{i-1}$};
   \draw (p+26) node[above] {\footnotesize $\alpha_i$};
   \path[every node/.style={font=\scriptsize}]
    (04) edge [bend right=15] node {} (p+16)
    (p+16) edge node {} (p+26)
    (p5) edge node {} (n4)
    (p5) edge node {} (p+16)
    (p+15) edge node {} (p+26)
    (04) edge [bend left=45] node {} (n4);                

\end{tikzpicture}
\caption{}
\label{GnH1}
\end{center}\end{figure}

Thus, the complexity of the computation comes down to determining $\#\Hom(\mathcal{H}_i^n\uparrow \mathcal{Q};\mathcal{B})$, i.e., the number of occurences of each quiver $\mathcal{H}_i^n$ in the Bratteli diagram $\mathcal{B}$. Note that $\mathcal{H}_i^n$ is exactly Figure 14 of \cite{sovII} and (4.1.7) of \cite{maslen} (Moreover, $\mathcal{Q}$ is exactly (4.9) in \cite{maslen}). Then by \cite{sovII,maslen}, 
$\# \Hom(\mathcal{H}_i^n\uparrow \mathcal{Q};\mathcal{B})$ is given by $$\begin{array}{lll}\displaystyle\sum_{\alpha_j,\beta_j\in\mathcal{B}^j}M(\beta_{n-1},\beta_{i-1})M(\beta_{i-1},\beta_{i-2})M(\alpha_{i},\alpha_{i-1})M(\alpha_{i},\beta_{i-1})M(\alpha_{i-1},\beta_{i-2})d_{\alpha_{i-1}}d_{\beta_{n-1}},\end{array}$$
where $M(\rho,\gamma)$ denotes the number of paths from $\gamma$ to $\rho$ in $\mathcal{B}$.

In Appendix \ref{Brauerapp} we use path counting in $\mathcal{B}$ to show 
$$\# \Hom(\mathcal{H}_i^n\uparrow \mathcal{Q};\mathcal{B})\leq \frac{16i-17}{2n-1}\dim(\mathcal{B}r_{n}).$$

Then by Lemma \ref{factorsum},
\begin{equation}\begin{array}{ll}t_{\mathcal{B}r_n}(R)&\displaystyle\leq t_{\mathcal{B}r_{n-1}}(R_{\mathcal{B}r_{n-1}})+2\sum_{i=2}^n\frac{16i-17}{2n-1}\\
&\displaystyle= t_{\mathcal{B}r_{n-1}}(R_{\mathcal{B}r_{n-1}}) + 2 \frac{(n-1)(8n-1)}{(2n-1)} \\
&\displaystyle\leq t_{\mathcal{B}r_1}(R_{\mathcal{B}r_{1}}) + 2 \sum_{i=2}^n \frac{(i-1)(8i-1)}{(2i-1)}\\
&\displaystyle \leq t_{\mathcal{B}r_1}(R_{\mathcal{B}r_{1}}) + (4n^2-n+3)\\
&=  4n^2-n+4.\end{array}\end{equation} \end{proof}

\subsection{The BMW Algebra}
The BMW algebra is a semisimple $\mathbb{C}(q,m,l)$-algebra that can be described in a similar manner to the Brauer algebra (see e.g. \cite{GoodmanHauschild}). Defined independently as the \textit{Kauffman tangle algebra} by Murakami \cite{murakami} and algebraically by Birman and Wenzl \cite{birmanwenzl}, it was shown in an unpublished paper by Wasserman \cite{wasser} that these two notions are equivalent, giving rise to the single BMW algebra. The Bratteli diagram for the BMW algebra is identical to that of the Brauer algebra \cite{hal}. Further, a natural basis, $\mathcal{B}_n=\{T_d\mid d\in Br_n\}$ for the BMW algebra is indexed by Brauer monoid elements. As such, Theorem \ref{brauerthm} extends to the BMW algebra:

\begin{theorem}[cf. Theorem \ref{BMWthm}]
The Fourier transform of an element $f$ in the BMW algebra $\mathcal{BMW}_n$ may be computed at a complete set $R$ of irreducible matrix representations of $\mathcal{BMW}_n$ adapted to the chain of algebras
$$\mathcal{BMW}_n> \mathcal{BMW}_{n-1}>\cdots> \mathcal{BMW}_0$$
in at most $(4n^2-n+4)\dim(\mathcal{BMW}_n)$
operations. 
\end{theorem}

\subsection{Background: The Temperley-Lieb Algebra} The Temperley-Lieb algebra, $\mathcal{T}_n$ is most easily defined as the subalgebra of the Brauer algebra generated by $\{id, e_1,\dots, e_{n-1}\}$ with relations inherited by the Brauer algebra: 
$$\begin{array}{llll}
 (1)& e_i^2=qe_i,&(2)&  e_ie_j=e_je_i,\;\; |i-j|>1\\
   (3)&e_ie_{i+1}e_i=e_i,\;\;\;\; e_{i+1}e_ie_{i+1}=e_{i+1}\\
 \end{array}$$
As a diagram algebra, $\mathcal{T}_n$ is generated by diagrams on $2n$ points connected by \textit{nonintersecting} lines. For more background and the equivalence of these two definitions, see \cite{ridout}. The dimension of $\mathcal{T}_n$ is given by the $n$th Catalan number (in \cite{ridout}, a bijection is demonstrated betweeh the set of generating diagrams of $\mathcal{T}_n$ and the set of increasing walks on $\mathbb{Z}^2$ from $(0,0)$, to $(n,n)$ which avoid crossing the diagonal). 

Figure \ref{TnBratt} shows the Bratteli diagram for the chain of Temperley-Lieb algebras $\mathcal{T}_4>\mathcal{T}_3>\mathcal{T}_2>\mathcal{T}_1>\mathcal{T}_0$. Note that we distinguish $\mathcal{T}_1$ from $\mathcal{T}_0$ only so that vertices at level $i$ correspond to representations of $\mathcal{T}_i$. Irreducible representations of $\mathcal{T}_i$ are indexed by partitions of $i$ with two or fewer parts, with an edge between $\rho\in\mathcal{T}^i$ and $\lambda\in\mathcal{T}^{i-1}$ if $\rho$ is obtained from $\lambda$ by adding a box \cite{goodmanwenzl}. Note that the Bratteli diagram of $\mathcal{T}_n$ is a subquiver of Young's lattice, the Bratteli diagram of the symmetric group $S_n$.

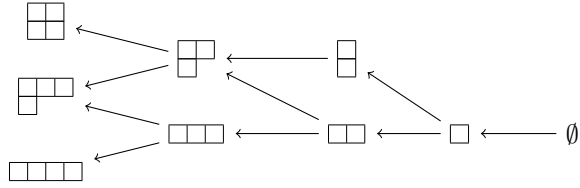
\begin{figure}[ht]\begin{center}
\begin{tikzpicture}[shorten >=1pt,node distance=2cm,on grid,auto,/tikz/initial text=] 

   \node at (-8,1.5) (40) {$\tiny{\yng(2,2)}$};
   \node at (-8,0.5) (41) {$\tiny{\yng(3,1)}$};
   \node at (-8,-.5) (42) {$\tiny{\yng(4)}$};
   \node at (-6,1) (33) {$\tiny{\yng(2,1)}$};
   \node at (-6,0) (32) {$\tiny{\yng(3)}$};
   \node at (-4,1) (23) {$\tiny{\yng(1,1)}$};
   \node at (-4,0) (22) {$\tiny{\yng(2)}$};
   \node at (-2.5,0) (1) {$\tiny{\yng(1)}$};
   \node at (-1,0) (0) {$\emptyset$};
   \path[every node/.style={font=\scriptsize}, ->]
    (0) edge  node {} (1)
    (33) edge node {} (40)
    (1) edge node {} (23)
    (1) edge node {} (22)
    (33) edge node {} (41)
    (32) edge node {} (41)
    (22) edge node {} (32)
    (22) edge node {} (33)
    (32) edge node {} (42)
    (23) edge node {} (33);
    
\end{tikzpicture}
\caption{Bratteli diagram for $\mathcal{T}_4>\mathcal{T}_3>\mathcal{T}_2>\mathcal{T}_1>\mathcal{T}_0$}
\label{TnBratt}
\end{center}\end{figure}

\subsection{Fourier transforms on $\mathcal{T}_n$}

 We first find a factor set for $\mathcal{T}_n$ over $\mathcal{T}_{n-1}$, then use the SOV approach to compute the Fourier transform of $f=\sum_{d\in T_n} f(d)d$, for $T_n$ the set of diagrams generating $\mathcal{T}_n$.

Let $E=\{e_{i}\cdots e_{n-1}\mid 1\leq i\leq n-1\}$
It follows immediately from Lemma \ref{facset} that $E\cup\{id\}$ is a factor set for $\mathcal{T}_n$ over $\mathcal{T}_{n-1}$. Note that this is also the factor set arising from the Jones Normal Form of elements in the Temperley-Lieb algebra (see e.g. Proposition 2.3 of \cite{ridout}).

\begin{theorem}[cf. Theorem \ref{templiebthm}]
The Fourier transform of an element $f$ in the Temperley-Lieb algebra $\mathcal{T}_n$ may be computed at a complete set $R$ of irreducible matrix representations of $\mathcal{T}_n$ adapted to the chain of algebras
$$\mathcal{T}_n> \mathcal{T}_{n-1}>\cdots> \mathcal{T}_0=\mathbb{C}(q)$$
in at most $\displaystyle\frac{n^3+9n^2+8n-12}{6}\dim(\mathcal{T}_n)$ 
operations. 
\end{theorem}
\begin{proof}
As noted above, $Y=E\cup\{id\}$ is a factor set for $\mathcal{T}_{n}$ over $\mathcal{T}_{n-1}$. For  $Y_i=\{id,e_i\}$, $Y\subseteq\{y_1y_2\cdots y_{n-1}\mid y_i\in Y_i\}$, giving a factorization of $Y$ as required by Step I of the SOV approach. Note that this is identical to the factorization in the proof of Theorem \ref{brauerthm}, with the only exception the size of $Y_i$. Thus, following the same steps in the SOV approach, the complexity of the computation comes down to determining $\#\Hom(\mathcal{H}_i^n\uparrow \mathcal{Q};\mathcal{B})$, i.e., the number of occurences of each quiver $\mathcal{H}_i^n$ of Figure \ref{GnH1} in the Bratteli diagram $\mathcal{B}$ associated to the chain of Temperley-Lieb algebras. 

Again by \cite{sovII,maslen}, 
$\# \Hom(\mathcal{H}_i^n\uparrow \mathcal{Q};\mathcal{B})$ is given by $$\begin{array}{lll}\displaystyle\sum_{\alpha_j,\beta_j\in\mathcal{B}^j}M(\beta_{n-1},\beta_{i-1})M(\beta_{i-1},\beta_{i-2})M(\alpha_{i},\alpha_{i-1})M(\alpha_{i},\beta_{i-1})M(\alpha_{i-1},\beta_{i-2})d_{\alpha_{i-1}}d_{\beta_{n-1}},\end{array}$$
where $M(\rho,\gamma)$ denotes the number of paths from $\gamma$ to $\rho$ in $\mathcal{B}$.

In Appendix \ref{Templiebapp} we use path counting in $\mathcal{B}$ to show 
$$\# \Hom(\mathcal{H}_i^n\uparrow \mathcal{Q};\mathcal{B})\leq \frac{(4i-6+2i^2)(n+1)(n)}{i(2n)(2n-1)}\dim(\mathcal{T}_{n}).$$

Then by Lemma \ref{factorsum},
\begin{equation}\begin{array}{ll}t_{\mathcal{T}_n}(R)&\displaystyle\leq t_{\mathcal{T}_{n-1}}(R_{\mathcal{T}_{n-1}})+\sum_{i=2}^n\frac{(4i-6+2i^2)(n+1)(n)}{i(2n)(2n-1)}\\ 
&\displaystyle\leq t_{\mathcal{T}_1}(R_{\mathcal{T}_{1}}) +  \sum_{i=2}^n \frac{i(i+5)(i+1)}{(4i-2)}\\ 
&\leq \displaystyle \frac{n^3+9n^2+8n-12}{6}.\end{array}\end{equation}

\end{proof}

\subsection{General Result}
We next give a general result (Theorem \ref{General}) to find efficient Fourier transforms on a finite dimensional semisimple algebra $A$ with special subalgebra structure. As the proof follows the same structure as the proof of Theorem \ref{brauerthm}, we leave it as an exercise.

Suppose 
$$A=A_n> A_{n-1}>\cdots> A_0,$$
is a chain of subalgebras of $A$ with subsets $Y_i\subseteq A_i$ such that
\begin{itemize}
\item[(1)] $Y_1=A_1$
\item[(2)] $ A_i\subseteq Y_2\cdots Y_i A_{i-1}$ for $2\leq i\leq n$.
\item[(3)] $Y_i$ commutes with $A_{i-2}.$
\end{itemize}
Note that the factor sets used for the Brauer algebra satisfied these three properties.

Let $\mathcal{B}$ be the Bratteli diagram associated to the chain 
$$A_n> A_{n-1}>\cdots>A_0,$$ and let $\{\mathbb{C}[\mathcal{B}_{i}]\}$ be the associated chain of path algebras. Let $$M(A_i,A_j):=\max M(\alpha_i,\alpha_j)$$ over all $\alpha_i\in\mathcal{B}^i, \alpha_j\in\mathcal{B}^j$ and let $\vert \hat{A}_i\vert$ denote the number of irreducible representations in a complete set of inequivalent irreducible representations of $A_i$.
\begin{theorem}\label{General}
Let $A_i$, $Y_i$ be as described above. Then the Fourier transform of an element $f\in A$ may be computed at a complete set $R$ of irreducible representations of $A_n$ adapted to the chain
$$A_n> A_{n-1}>\cdots> A_0$$
in at most $$\dim(A_n) \sum_{k=1}^n \sum_{i=2}^k M(A_{i-1},A_{i-2})^2\vert \hat{A}_{i-2}\vert \frac{\dim(A_i)}{\dim(A_{i-1})}\frac{\dim(A_{k-1})}{\dim(A_k)} \prod_{j=i}^k \vert B_j\vert$$ operations.
\end{theorem}

\section{Further Directions}\label{conclusion}

In this paper we extended the SOV approach of \cite{sovi,sovII} to the semisimple algebra setting and provided the first known complexity upper bounds for Fourier transforms on the Brauer, BMW, and Temperley-Lieb algebras.

Efficiency counts are determined by the choice of factor sets, size of the factorization sets of these factor sets, and the number of occurrences of the corresponding subgraphs in the Bratteli diagram. While the choice of factor set for the Brauer algebra (Lemma \ref{facset}) is easy to describe, it is by no means `canonical'. On the other hand, while the choice of factor set for the Temperley-Lieb algebra is canonical in that it comes from Jones Normal Form, the bound of Theorem \ref{templiebthm} is worse than anticipated given that the Temperley-Lieb algebra is a subalgebra of the Brauer algebra. Future directions could explore different choices of factor sets and the bounds they provide, as well as more refined path-counting. 

The examples in this paper only touch on the wealth of semisimple algebras whose structure and Bratelli diagrams are known. In \cite{grood}, Grood constructs the irreducible representations of the rook partition algebra and the associated Bratteli diagram, while Halverson et al. \cite{halverson, halver} determine analogues of the seminormal representations of $S_n$ for the rook-Brauer algebra and planar-rook algebra. It is an interesting and ongoing project to extend and apply the results of this paper to these other examples by developing an understanding of the centralizers and irreducible representations and to explore the resulting combinatorial path-counting questions to provide efficient counts. 


\appendix

\section{Brauer Algebra Combinatorial Lemmas}\label{Brauerapp}

Let $\mathcal{B}$ denote the Bratteli diagram associated to the chain of Brauer algebras $\mathcal{B}r_n>\mathcal{B}r_{n-1}>\cdots \mathcal{B}r_1>\mathcal{B}r_0$ (Figure \ref{BrnBratt}). 
The following two lemmas provide a bound for $\#\Hom (\mathcal{H}_i^n\uparrow G;\mathcal{B})$, for $\mathcal{H}_i^n$ as in Figure \ref{GnH1}.
\begin{lemma}\label{BrnH1}
\begin{itemize}
\item[]
\item[(1)] $\displaystyle\#\Hom(\mathcal{H}_i^n\uparrow G;\mathcal{B})=\frac{ \dim(\mathcal{B}r_{n-1})}{\dim(\mathcal{B}r_{i-1})}\#\Hom(\mathcal{H}_i^i\uparrow G;\mathcal{B}),$
\item[(2)] $\#\Hom(\mathcal{H}_i^i\uparrow G;\mathcal{B})$
$$\leq 2\displaystyle\frac{\dim(\mathcal{B}r_{i-1})^2}{\dim(\mathcal{B}r_{i-2})}+\sum_{\beta_{i-1}\in\mathcal{B}^{i-1}}(4\jmp(\beta_{i-1})^2+2\jmp(\beta_{i-1})+1)(d_{\beta_{i-1}})^2,$$
\end{itemize} where $\jmp$ denotes the jump of a partition, i.e, the number of ways to remove a single box to form a new partition.
\end{lemma}
\begin{proof}
Part (1) has the same proof as Lemma D.3 in \cite{sovII}.

To prove (2), consider 
$$\begin{array}{l}\#\Hom(\mathcal{H}_i^i\uparrow G;\mathcal{B})\\
=\displaystyle\sum_{\alpha_j,\beta_j\in\mathcal{B}^j}M_\mathcal{B}(\beta_{i-1},\beta_{i-2})M_\mathcal{B}(\alpha_{i},\alpha_{i-1})M_\mathcal{B}(\alpha_{i},\beta_{i-1})M_\mathcal{B}(\alpha_{i-1},\beta_{i-2})d_{\beta_{i-1}}d_{\alpha_{i-1}}\\
=\displaystyle\sum_{\alpha_{i-1}\neq\beta_{i-1}}+\sum_{\alpha_{i-1}=\beta_{i-1}},\end{array}$$
for $\displaystyle\sum_{\alpha_{i-1}\neq\beta_{i-1}}$ the sum
$$\begin{array}{l}\displaystyle\sum_{\substack{\alpha_j,\beta_j\in\mathcal{B}^j\\\alpha_{i-1}\neq\beta_{i-1}}}M_\mathcal{B}(\beta_{i-1},\beta_{i-2})M_\mathcal{B}(\alpha_{i},\alpha_{i-1})M_\mathcal{B}(\alpha_{i},\beta_{i-1})M_\mathcal{B}(\alpha_{i-1},\beta_{i-2})d_{\beta_{i-1}}d_{\alpha_{i-1}}\end{array}$$
and $\displaystyle\sum_{\alpha_{i-1}=\beta_{i-1}}$ the sum
$$\begin{array}{l}\displaystyle\sum_{\substack{\alpha_j,\beta_j\in\mathcal{B}^j\\\alpha_{i-1}=\beta_{i-1}}}M_\mathcal{B}(\beta_{i-1},\beta_{i-2})^2M_\mathcal{B}(\alpha_{i},\beta_{i-1})^2(d_{\beta_{i-1}})^2.\end{array}$$

First suppose $\alpha_{i-1}$ and $\beta_{i-1}$ are distinct partitions. Then they jointly determine $\alpha_i$ up to two choices. This is clear if $\alpha_{i-1}$ and $\beta_{i-1}$ both partition $k$, as they then jointly determine exactly one partition of $k+1$ and one partition of $k-1$. Now suppose, without loss of generality, that $\alpha_{i-1}$ is a partition of $k$ while $\beta_{i-1}$ is a partition of $k-2$. Then to both be connected to a vertex, $\alpha_i$, at level $i$, $\beta_{i-1}$ must be obtained from $\alpha_{i-1}$ by removing two boxes, which can only be done in two ways.

Then as in the proof of Lemma D.3 of \cite{sovII},
\begin{equation}\label{Brone}\sum_{\alpha_{i-1}\neq\beta_{i-1}}\leq 2\left(\frac{\dim(\mathcal{B}r_{i-1})^2}{\dim(\mathcal{B}r_{i-2})}-\sum_{}\jmp(\beta_{i-1})(d_{\beta_{i-1}})^2\right). \end{equation}

Now suppose $\alpha_{i-1}=\beta_{i-1}$. Then $\alpha_i$ is obtained from $\beta_{i-1}$ by either adding or removing a box, and similarly for $\beta_{i-2}$. Thus, \begin{equation}\label{Brtwo}\sum_{\alpha_{i-1}=\beta_{i-1}}=\sum_{\beta_{i-1}\in\mathcal{B}^{i-1}}(2\jmp(\beta_{i-1})+1)(2\jmp(\beta_{i-1})+1)(d_{\beta_{i-1}})^2.\end{equation}

Summing equations (\ref{Brone}) and (\ref{Brtwo}) gives part (2).

\end{proof}
Combining Lemma \ref{BrnH1} with the fact that $\jmp(\beta_i)^2\leq 2i$ (see proof of \cite{maslen}[Lemma 5.3]) gives the following bound:
\begin{corollary}\label{BrnH3} 
$\#\Hom(\mathcal{H}_i^n\uparrow G;\mathcal{B})\leq \frac{16i-17}{2n-1}\dim(\mathcal{B}r_n)$.
\end{corollary}

\section{Temperley-Lieb Algebra Combinatorial Lemmas}\label{Templiebapp}

Let $\mathcal{B}$ denote the Bratteli diagram associated to the chain of Temperley-Lieb algebras $\mathcal{T}_n>\mathcal{T}_{n-1}>\cdots \mathcal{T}_1>\mathcal{T}_0$ (Figure \ref{TnBratt}). 
The following two lemmas provide a bound for $\#\Hom (\mathcal{H}_i^n\uparrow G;\mathcal{B})$, for $\mathcal{H}_i^n$ as in Figure \ref{GnH1}.
\begin{lemma}\label{TnH1}
\begin{itemize}
\item[]
\item[(1)] $\displaystyle\#\Hom(\mathcal{H}_i^n\uparrow G;\mathcal{B})=\frac{ \dim(\mathcal{T}_{n-1})}{\dim(\mathcal{T}_{i-1})}\#\Hom(\mathcal{H}_i^i\uparrow G;\mathcal{B}),$
\item[(2)] $\#\Hom(\mathcal{H}_i^i\uparrow G;\mathcal{B})\leq \displaystyle\frac{\dim(\mathcal{T}_{i-1})^2}{\dim(\mathcal{T}_{i-2})}+\sum_{\beta_{i-1}\in\mathcal{B}^{i-1}}(\jmp(\beta_{i-1})^2(d_{\beta_{i-1}})^2.$
\end{itemize} 
\end{lemma}
\begin{proof}
This is exactly Lemma 5.2 of \cite{maslen}, replacing the order of the symmetric group with the dimension of the Temperley-Lieb algebra.
\end{proof}
Combining Lemma \ref{BrnH1} with the fact that $\jmp(\beta_i)^2\leq 2i$ 
\begin{corollary}\label{TnH3} 
$\#\Hom(\mathcal{H}_i^n\uparrow G;\mathcal{B})\leq \frac{(4i-6+2i^2)(n+1)(n)}{i(2n)(2n-1)}\dim(\mathcal{T}_n)$.
\end{corollary}

\bibliographystyle{abbrv}
\bibliography{SOVIIbib}

\end{document}